\documentclass[graybox]{svmult}
\usepackage{mathptmx}       % selects Times Roman as basic font
\usepackage{helvet}         % selects Helvetica as sans-serif font
\usepackage{courier}        % selects Courier as typewriter font
\usepackage{type1cm}        % activate if the above 3 fonts are
\usepackage{makeidx}         % allows index generation
\usepackage{graphicx}        % standard LaTeX graphics tool
\usepackage{multicol}        % used for the two-column index
\usepackage[bottom]{footmisc}% places footnotes at page bottom

\usepackage{verbatim}        % for "`comment"'
\usepackage{amsmath}

\usepackage{amsfonts}

\usepackage{enumitem}        % generalizes list enviroments

\newcommand{\C}{\mathbb{C}}
\newcommand{\R}{\mathbb{R}}
\newcommand{\N}{\mathbb{N}}
\newcommand{\Z}{\mathbb{Z}}

\renewcommand{\d}[1][x]{\,\operatorname{d}\!#1}
\newcommand{\ddt}{\frac{\d[]}{\d[t]}}
\newcommand{\ddtddt}{\frac{\d[]^2}{\d[t^2]}}
\newcommand{\intRd}{\int_{\R^d}}
\newcommand{\Rd}{\R^d}

\newcommand{\spn}{\operatorname{span}}
\newcommand{\ran}{\operatorname{ran}}

\newcommand{\tr}{\operatorname{Tr}}

\newcommand{\A}{\mathbf{A}}
\newcommand{\B}{\mathbf{B}}
\newcommand{\CC}{\mathbf{C}}
\newcommand{\DD}{\mathbf{D}}
\newcommand{\II}{\mathbf{I}}

\newcommand{\LL}{\mathbf{L}}

\renewcommand{\P}{\mathbf{P}}
\newcommand{\Q}{\mathbf{Q}}
\newcommand{\RR}{\mathbf{R}}

\newcommand{\ip}[2]{\langle {#1}\,,\, {#2} \rangle}
\newcommand{\ipBig}[2]{\Big\langle {#1}\,,\, {#2} \Big\rangle}
\newcommand{\norm}[1]{\| {#1} \|}

\newcommand{\normP}[1]{\| {#1} \|_{\P}}
\DeclareMathOperator{\diag}{diag}

\newcommand{\T}{\mathbf{T}}

\begin{document}

\title*{On linear hypocoercive BGK models}
\numberwithin{equation}{section} % - does not work yet!
\author{Franz Achleitner, Anton Arnold, Eric A.\ Carlen}
\institute{Franz Achleitner \at Vienna University of Technology, Institute of Analysis and Scientific Computing, Wiedner Hauptstr. 8-10, A-1040 Wien, Austria, \email{franz.achleitner@tuwien.ac.at}
\and Anton Arnold \at Vienna University of Technology, Institute of Analysis and Scientific Computing, Wiedner Hauptstr. 8-10, A-1040 Wien, Austria, \email{anton.arnold@tuwien.ac.at}
\and Eric A.\ Carlen \at Department of Mathematics, Rutgers University,
110 Frelinghuysen Rd., Piscataway NJ 08854, USA,
 \email{carlen@math.rutgers.edu}}
\maketitle

\abstract{We study hypocoercivity for a class of linear and linearized BGK models for discrete and continuous
phase spaces. We develop methods for constructing entropy functionals that prove
exponential rates of relaxation to equilibrium. Our strategies are based on the entropy and spectral methods, adapting Lyapunov's direct method (even for ``infinite matrices'' appearing for continuous phase spaces) to construct appropriate entropy functionals. Finally, we also prove local asymptotic stability of a nonlinear BGK model.}

%%%%%%%%%%%%%%%%%%%%%%%%%%%%%%%%%%%%%%%%%%%%%%%%%%%%%%%%%%%%%%%%%%%%%%%%%%%%%%%%%

\section{Introduction}
\label{sec:1}

This paper is concerned with the large time behavior of linear BGK models (named after the physicists Bhatnagar-Gross-Krook \cite{BGK54}) for a phase space density $f(x,v,t)$; $x,\,v\in\Rd$, satisfying the kinetic evolution equation
\begin{equation}\label{bgk}
  f_t+v\cdot\nabla_x f-\nabla_x V \cdot \nabla_v f = \Q f := M_{T(t)}(v)\,\intRd f(x,v,t)\,\d[v] -f(x,v,t)\,,\quad t\ge0\,,
\end{equation}
with some given confinement potential $V(x)$ and where
 $M_{T}$ denotes the normalized Maxwellian at  some temperature $T$:
 \begin{equation*}
 M_T(v) = (2\pi T)^{-d/2}e^{-|v|^2/2T}\ .
\end{equation*}
We assume that the initial condition is normalized as 
$$
  \int_{\Rd\times\Rd} f(x,v,0)\,{\rm d} x {\rm d}v = 1\ ,
$$
and this normalization persists under the flow of \eqref{bgk}.
%In \eqref{bgk} 
The function $T(t)$ is defined so that the energy is conserved:
$$\int_{\Rd\times\Rd} \left[\frac{|v|^2}{2} + V(x)\right] f(x,v,t)\,{\rm d} x {\rm d}v = 
\int_{\Rd\times\Rd} \left[\frac{|v|^2}{2} + V(x)\right] f(x,v,0)\,{\rm d} x {\rm d}v =: E_0\ .$$
This is achieved in case
\begin{equation}\label{Ttime}
T(t) := \frac2d\left[ E_0 - \int_{\Rd} V(x)\rho(x,t)\,{\rm d }x\right]\ ,
\end{equation}
where $\rho(x,t) := \int_{\Rd}f(x,v,t)\,{\rm d}v$, which completes the specification of the equation.

This model differs form the usual BGK model in that the
Maxwellian $M_T$ has a spatially constant temperature and zero momentum. This is already a simplification of
the standard BGK model in which $M_T$ would be replaced by the local Maxwellian corresponding to $f$; i.e., the
local Maxwellian with the same hydrodynamic moments as $f$. However, \eqref{bgk}-\eqref{Ttime} is still non-linear since $T(t)$ depends linearly
on $f$, but then $M_T$ depends nonlinearly on $T$. This simplified 
equation arises in certain models of thermostated systems \cite{BL}.  
Under sufficient growth assumptions on $V$ as $|x|\to\infty$, the unique normalized steady state of \eqref{bgk} is
$$
  f^\infty(x,v) = \exp\left(-\frac{1}{T _\infty}\big[V(x)+\frac{|v|^2}{2}\big]\right)\,,
$$
where the normalization constant shall be included in $V$ and $T_\infty$ such that the energy associated to $f^\infty$ is $E_0$.

In fact, we simplify the model further: %a further simplification is studied in this context:
We take $d=1$, replace the spatial domain $\Rd$ by the unit circle $\mathbf{T}^1$, and then dispense with the confining potential.
Thus we shall first investigate the linear BGK model
\begin{equation}\label{bgk2}
  f_t+v\ f_x  = \Q f := M_{T}(v)\,\int_\R f(x,v,t)\,\d[v] -f(x,v,t)\,,\quad t\ge0\,.
\end{equation}
Let $\d[\tilde x]$ denote the normalized Lebesgue measure on $\mathbf{T}^1$, and consider normalized initial data $f(x,v,0)$
 such that $\int_{\mathbf{T}^1\times\R} f(x,v,0)\,{\rm d} \tilde x {\rm d}v=1$ (a normalization which is conserved under the flow).
In this case, equation (\ref{Ttime}) for the temperature reduces to %$T(t) = (2/d)E_0$, 
$T(t) = 2 E_0$, independent of $t$, with
$$
  E_0:=\int_{\mathbf{T}^1\times\R} \frac{v^2}{2}f(x,v,0)\,{\rm d} \tilde x {\rm d}v \ .
$$

For the simplified linear equation~\eqref{bgk2}, the unique steady state is $f^\infty = M_T$, uniform on the circle. 
%provided we normalize so that the total mass (also conserved) is one. 
We shall study the rate at which normalized solutions of (\ref{bgk2}) approach the steady state $f^\infty = M_T$ as $t\to\infty$. 
This problem is interesting since the collision mechanism drives the local velocity distribution towards $M_T$, but  a more complicated mechanism
involving the interaction of the streaming term $v \partial_x$ and the collision operator $\Q$  is responsible for the emergence of spatial uniformity. 

To elucidate this key point, let us define the operator ${\bf L}$ by
$$ {\bf L} f(x,v) := -v\ \partial_x f (x,v) + \Q f(x,v)\ .$$
Then the evolution equation (\ref{bgk2}) can be written $f_t = {\bf L}f$.  Let $\mathcal{H}$ denote the weighted space $L^2(\mathbf{T}^1\times\R;M_T^{-1}(v)\d[v])$.
Then ${\bf Q}$ is self-adjoint on $\mathcal{H}$, ${\bf L}f^\infty = 0$, and a simple computation shows that if $f(t)$ is a solution of (\ref{bgk2}),
$$
\frac{{\rm d}}{{\rm d}t} \| f(t) - f^\infty\|_\mathcal{H}^2 = 2 \langle f(t), {\bf L} f(t)\rangle_\mathcal{H} = 2 \langle f(t), \Q f(t)\rangle_\mathcal{H}
= -2\| f  - M_T \rho\|_{\mathcal{H}}^2\ ,
$$
where, as before, $\rho(x,t) := \int_{\R}f(x,v,t)\,{\rm d}v$. Thus, while the norm $ \| f(t)- f^\infty\|_\mathcal{H}$ is monotone decreasing, the derivative is zero
whenever $f(t)$ has the form $f(t) = M_T \rho$ for {\em any} smooth density $\rho$. In particular, the inequality
\begin{equation}\label{coercive}
\langle f - f^\infty , {\bf L} (f - f^\infty) \rangle_\mathcal{H} \leq - \lambda \|f - f^\infty\|_\mathcal{H}^2
\end{equation}
is valid in general for $\lambda = 0$, but for no positive value of $\lambda$.  If (\ref{coercive}) were valid for some $\lambda>0$, we would have had
$\|f(t) - f^\infty\|_\mathcal{H}^2 \leq e^{-t\lambda}\|f(0) - f^\infty\|_\mathcal{H}^2$ for all solutions of our equation, and we would say that
the evolution equation is {\em coercive}. However, while this is not the case, it does turn out that one still has constants $1< C < \infty$ and $\lambda> 0$
such that
\begin{equation}\label{coercive2}
\|f(t) - f^\infty\|_\mathcal{H}^2 \leq Ce^{-t\lambda}\|f(0) - f^\infty\|_\mathcal{H}^2\ .
\end{equation}
(The fact that there exist initial data $f(0) \neq f^\infty $ for which the derivative of the norm is zero shows that necessarily $C>1$.) In Villani's terminology
(see \S3.2 of \cite{ViH06}),
this means that our evolution equation is {\em hypocoercive}.

%
%In this regard we mention that the collision operator $-\Q$ is positive but \emph{not} coercive on the weighted space $L^2(2\pi\times\R;M_T^{-1})$  (choose $f=\sin(x)M_T(v)$, e.g.). 
%Nevertheless we shall show that the convergence rate is exponential, so that our system is {\em hypocoercive} on $\{f^\infty\}^\perp\subset L^2(2\pi\times\R;M_T^{-1})$. 

Many hypocoercive equations have been studied in recent years \cite{ViH06, He06, DoMoScH09, DoMoScH10, ArEr14},
 including BGK models in \S 1.4 and \S 3.1 of \cite{DoMoScH10} (see also \S\ref{sec:41} below),
 but sharp decay rates were rarely an issue there.
The fact that normalized solutions of (\ref{bgk2}) converge exponentially fast at {\em some}
rate to $f^\infty$ 
is a consequence of a probabilistic analysis of such equations in \cite{BL}: In fact, equation (\ref{bgk2}) is the Kolmogorov forward equation for a certain Markov process, and as shown in  \cite{BL} an argument based on a Doeblin condition yields exponential convergence. However, this approach relies on compactness arguments and does not yield explicit values for $C$ or $\lambda$.  We shall discuss another approach to the problem of establishing hypocoercivity for such models
that does yield explicit -- and quite reasonable -- values for $C$ and $\lambda$. 
To this end, our main tool will be variants of the \emph{entropy--entropy production method.}
Our first main result will be a decay estimate for \eqref{bgk2}:
\begin{theorem}\label{bgk-decay}
{\bf [decay estimate for \eqref{bgk2}]} Fix unit temperature $T= 1$. There exists an entropy functional $e(f)$ satisfying
%$$2 \| f - M_1\|^2_{L^2(M_1^{-1})}  \geq e(f) \geq \frac14 \| f - M_1\|^2_{L^2(M_1^{-1})}$$ 
$$  \frac12 e(f) \leq \| f - M_1\|^2_\mathcal{H} \le 4e(f) $$ 
such that for all (normalized) solutions $f(t)$ of (\ref{bgk2}) with $e(f^I) < \infty$,
$$e(f(t)) \leq e^{-t\cdot 0.547592...}e(f^I)\ ,\qquad t\ge0\ .$$
\end{theorem}
\bigskip

% Franz 18.12.2015
Finally, we shall study the linearization of a one dimensional BGK equation around a Maxwellian with some constant-in-$x$ temperature. %As before we consider this problem on the one dimensional torus $\T^1$. 
% Let $f(x,v)$ be a probability density on $\T^1\times\R$. with zero (macroscopic) momentum, i.e. $U(x) := \int_\R v f(x,v) \d[v]=0$ for all $x$. 
In one dimension, if collisions conserve both energy and momentum, they are trivial: The only kinematic possibilities are an exchange of velocities which has no effect at all at the kinetic level. Therefore, in one dimension the natural BGK equation, which would correspond for example to the Kac equation \cite{K56}, uses Maxwellians determined by the density and temperature alone. The method will be applied to the three dimensional equation in a follow-up paper.

For a probability density $f(x,v)$ on $\T^1\times\R$ we thus consider  the nonlinear BGK equation
\begin{equation} \label{bgk:torus} 
 f_t(x,v,t) +  v\ f_x(x,v,t) = M_f(x,v,t) -  f(x,v,t)\ , \qquad t\geq 0\ , 
\end{equation}
where $M_f$ is the local Maxwellian having the same local density and ``temperature'' as $f$: The density is defined as $\rho(x,t) := \int_\R f(x,v,t)\d[v]$ % and the (macroscopic) velocity $U(x) = \int_\R v f(x,v) \d[v]$. 
%Assuming that the (macroscopic) velocity $U(x)=0$ yields identities for 
and the pressure as $P(x,t) := \int_\R v^2f(x,v,t)\d[v]$. In analogy to the situation with zero velocity we shall refer to the conditional second moment, $\tilde T(x,t) := P(x,t)/\rho(x,t)$ as temperature (with the gas constant scaled as $R=1$). 
Then, for fixed $t$, the local Maxwellian $M_f$ is defined as
\begin{equation}\label{mloc}
M_f(x,v) = \frac{\rho(x)}{\sqrt{2\pi \tilde T(x)}} e^{-v^2/2\tilde T(x)} = \frac{\rho^{3/2}(x)}{\sqrt{2\pi P(x)}}e^{-v^2\rho(x)/2P(x)}\ ,
\end{equation}
and we shall mostly use the second version of it in the sequel.
% Franz 15.1.2016
The existence of global solutions for the Cauchy problem of similar nonlinear BGK models has been proven in \cite{Pe89, PePu93, BoCa09}.
%Consider the Cauchy problem for~\eqref{bgk:torus}-\eqref{mloc} with an initial datum $f_0(x,v)$.
%If $f_0(x,v)$ is a probability density on $\T\times\R$
% then the solution $f(x,v,t)$ is again a probability density on $\T\times\R$ for all $t\geq 0$.
%Moreover, hydrodynamic moments will be conserved?!
%Therefore, if $\int_\R v f_0(x,v) \d[v] =0$ then $U(x) = \int_\R v f(x,v,t) \d[v] = 0$?!

%Let $\d[\tilde x]$ denote the normalized Lebesgue measure on $\T^1$. 
We assume $\int_{\T^1} \rho(x)\d[\tilde x] = 1$ and define $T:=\int_{\T^1} P(x)\d[\tilde x]$, which are both conserved by the flow of \eqref{bgk:torus}.
Now we consider $f$ close to the global equilibrium $M_T(v)$, with  $h$ defined by $f = M_T+h$. Then 
%\begin{equation}\label{f:perturb}
%  \rho(x,t) = 1 + \sigma(x,t) \qquad{\rm and}\qquad P(x,t) = T + \tau(x,t)\ , 
%\end{equation}
\begin{align}
\begin{split}\label{f:perturb}
  \rho(x,t) &= 1 + \sigma(x,t) \qquad{\rm with}\qquad \sigma(x,t):=\int_\R h(x,v,t) \d[v]\ ,\\
  P(x,t) &= T + \tau(x,t) \qquad{\rm with}\qquad \tau(x,t):=\int_\R v^2h(x,v,t) \d[v]\ , 
\end{split}
\end{align}
which implies
\begin{equation}\label{mloc2}
\int_{\T^1} \sigma(x,t)\d[\tilde x] =0  \qquad{\rm and}\qquad   \int_{\T^1} \tau(x,t)\d[\tilde x] =0\ .
\end{equation}
%Define $h$ by $f = M_T+h$.  
The perturbation $h$ then satisfies
$$
 h_t (x,v,t) +  v\ h_x(x,v,t) = [M_f(x,v,t)  -M_T(v)]  - h(x,v,t)\ , \quad t\geq 0 \ .
$$
For $\sigma$ and $\tau$ small we have
\begin{align}\label{taylor}
 M_f(x,v) - M_T(v) &= \frac{(1+\sigma)^{3/2}(x)}{\sqrt{2\pi (T+\tau(x))}}e^{-v^2(1+\sigma(x))/2(T+\tau(x))} - \frac{1}{\sqrt{2\pi T}}e^{-v^2/2T}\\
  &\approx \left( \frac32 - \frac{v^2}{2T}\right)M_T(v)\sigma(x)
      + \left( -\frac{1}{2T}+ \frac{v^2}{2T^2}\right)M_T(v)\tau(x) \ ,
\end{align}
which yields the linearized BGK model that we shall analyze in this paper:
\begin{align} \label{linBGK:torus}
 & h_t (x,v,t) +  v\ h_x(x,v,t) \\
 &\quad = M_T(v) \,\left[\left( \frac32 - \frac{v^2}{2T}\right)\sigma(x,t)
      + \left( -\frac{1}{2T}+ \frac{v^2}{2T^2}\right)\tau(x,t)\right]  - h(x,v,t)\ , \quad t\geq 0 \ .\nonumber
\end{align}
Following the same approach as for Theorem \ref{bgk-decay} we shall obtain a decay estimate for \eqref{linBGK:torus}, and then local asymptotic stability for the nonlinear BGK equation (\ref{bgk:torus}).  For the latter purpose, we need to introduce another set of norms.

For $\gamma \geq 0$, let $H^\gamma(\mathbf{T}^1)$ be the Sobolev space consisting of  the completion  of smooth functions $\varphi$ on $\mathbf{T}^1$ in the Hilbertian norm 
$$\|\varphi\|_{H^\gamma}^2 := \sum_{k\in \Z} (1+k^2)^{\gamma}|\varphi_k|^2\ ,$$
where $\varphi_k$ is the $k$th Fourier coefficient of $\varphi$. Let $\mathcal{H}_\gamma$ denote the Hilbert space
$H^\gamma(\mathbf{T}^1)\otimes L^2(\R;M_T^{-1})$.  
%$H^\gamma(\mathbf{T}^1)\otimes L^2(\R;M_T^{-1}(v){\rm d}v)$.  
Then the inner product in $\mathcal{H}_\gamma$ is given by
$$\langle f,g\rangle_{\mathcal{H}_\gamma} = \int_{\mathbf{T}^1} \int_\R \overline{f}(x,v) \left[\left(1 - \partial_x^2\right)^{\gamma} g(x,v) \right]M_T^{-1}(v){\rm d}v {\rm d}\tilde x \ . $$

\begin{theorem}\label{linBGK-decay}
{\bf [decay estimates for \eqref{linBGK:torus}, \eqref{bgk:torus}]} 
Fix unit temperature $T=1$.
\begin{enumerate}
\item[(a)]
For all $\gamma \ge 0$ there is an entropy functional $e_\gamma(f)$ satisfying
\begin{equation}\label{comp}
\frac23 e_\gamma(f)  \leq \|f-M_1\|_{\mathcal{H}_\gamma}^2 \leq \frac43 e_\gamma(f)
\end{equation}
such that if $h=f-M_1$ is a solution of the linearized BGK equation \eqref{linBGK:torus} with initial data $h^I=f^I-M_1$ such that $\int_{\T^1}\int_{\R} (1,\,v^2)\,f^I\,\d[v]\d[\tilde x]=(1,\,1)$, 
and ${e_\gamma(f^I)< \infty}$, then
\begin{equation}\label{locasstab5}
e_\gamma(f(t)) \leq e^{-t/25}e_\gamma(f^I)\ ,\qquad t\ge0\ .
\end{equation}

\item[(b)]
Moreover, for all $\gamma>1/2$, there is an explicitly computable  $\delta_\gamma>0$ such that 
if $f$ is a solution of the nonlinear BGK equation \eqref{bgk:torus} with initial data $f^I$ such that $\int_{\T^1}\int_{\R} (1,\,v^2)\,f^I\,\d[v]\d[\tilde x]=(1,\,1)$, and 
$\|f^I-M_1\|_{\mathcal{H}_\gamma} < \delta_\gamma$, then for the same entropy function 
$e_\gamma$, \eqref{locasstab5} is again valid. 
\end{enumerate}
\end{theorem}

\bigskip
%This paper is organized as follows. 
Before turning to our main investigation, i.e.~exponential decay in the BGK equations \eqref{bgk2}, \eqref{linBGK:torus}, \eqref{bgk:torus}, we shall study some still simpler models with a finite number of positions and velocities: In \S\ref{sec:2} we analyze coercive BGK models with first two and then finitely many velocities using relative entropies. Since this approach fails for discrete hypocoercive BGK models (considered in \S\ref{sec:3}), their analysis will be based on spectral methods and Lyapunov's direct method. \S\ref{sec:4} is concerned with space-inhomogeneous BGK models. We shall start with its discrete velocity analogs in \S\ref{sec:41}--\S\ref{sec:42}, where the velocity modes will be expanded in Krawtchouk polynomials -- a discrete analog of the Hermite polynomials. In section \ref{sec:43} we shall finally analyze the exponential convergence of the linear BGK equation \eqref{bgk2}, using a Hermite expansion of the velocity modes and an adaption of Lyapunov's!
  direct method, used here for 
``infinite matrices''. This will yield the proof of Theorem \ref{bgk-decay}. This strategy is modified in \S\ref{sec:44} for the linearized BGK equation \eqref{linBGK:torus}, proving Theorem \ref{linBGK-decay}(a). Finally, in \S\ref{sec:45} we analyze the local asymptotic stability of the nonlinear BGK equation \eqref{bgk:torus}, as stated in Theorem \ref{linBGK-decay}(b).

%%%%%%%%%%%%%%%%%%%%%%%%%%%%%%%%%%%%%%%%%%%%%%%%%%%%%%%%%%%%%%%%%%%%%%%%%%%%%%%%%

\section{Discrete coercive BGK models}
\label{sec:2}

%Before turning to our main investigation we study some still simpler models with a finite number of positions and velocities. These are introduced in the next section.

In this section we consider space-homogeneous BGK models with a finite number of velocities. 
Our main tool in the investigation is the relative entropy, which is defined as follows
(see \S2.2 of \cite{ArMaToUn01} for more details):
\begin{definition}\label{def-rel-entr}
\begin{enumerate}
  \item [(a)] Let $J$ be either $\R^+$ or $\R$. A scalar function $\psi\in C(\bar J)\cap C^2(J)$ satisfying the conditions
  \begin{equation}\label{entropy-generator}
    \psi(1)=0\,,\quad\psi\ge0\,,\quad \psi''\ge0\,,
%    \quad (\psi''')^2\le\frac12\psi''\psi^{IV}
    \;\mbox{ on } J 
  \end{equation}
  (and hence also $\psi'(1)=0$) is called \emph{entropy generator}.
  \item [(b)]
  Let $f_1\in L^1(\R^{2d})$, $f_2\in L^1_+(\R^{2d})$ with $\int\!\int f_1\d\d[v]=\int\!\int f_2\d\d[v]=1$ and $\frac{f_1}{f_2}(x,v)\in \bar J$ a.e.\ (w.r.t.\ the measure $f_2(\d\d[v])$). Then
  \begin{equation}\label{rel-entropy}
    e_\psi(f_1|f_2):=\int\!\!\int_{\R^{2d}} \psi\Big(\frac{f_1}{f_2}\Big) f_2 \d\d[v]\ge0
  \end{equation}
  is called %an \emph{admissible relative entropy} 
  a \emph{relative entropy}  of $f_1$ with respect to $f_2$ with generating function $\psi$.
\end{enumerate}
\end{definition}

%In this definition, the term ``admissible'' refers to the applicability of the \emph{entropy method} under the assumptions \eqref{entropy-generator}.
In applications, the most important examples are the logarithmic entropy $e_1(f_1|f_2)$, generated by
$$
  \psi_1(\sigma):=\sigma\ln\sigma-\sigma+1\,,
$$
and the power law entropies $e_p(f_1|f_2)$,  generated by
\begin{equation}\label{p-entropy}
  \psi_p(\sigma):=\sigma^p-1-p(\sigma-1)\,,\quad p>1\,.
\end{equation}
Except for the quadratic entropy $e_2$ we shall always use $J=\R^+$.
Below we shall use also a second family of power law entropies $\hat e_p(f_1|f_2)$ generated by
\begin{equation}\label{p-entropy2}
 \hat \psi_p(\sigma):=|\sigma-1|^p\,,\quad p>1\,.
\end{equation}

The above definition clearly shows that $e_\psi(f_1|f_2)=0$ iff $f_1=f_2$. In the next section we shall hence try to prove that solutions $f(t)$ to BGK models satisfy $e_\psi(f(t)|f^\infty)\to0$ as $t\to\infty$. For the entropies $e_p,\,p\ge1$
such a convergence in relative entropy then also implies $L^1$--convergence, due to the \emph{Csisz\'ar-Kullback inequality}: 
$$
  \|f_1-f_2\|_{L^1(\R^{2d})}^2 \le 2\,e_1(f_1|f_2) \le \frac{2}{p(p-1)}\,e_p(f_1|f_2)\,,
$$
where we used $\psi_1(\sigma) \le\psi_p(\sigma)/\psi''_p(1),\,\sigma\ge0$ in the second inequality.
%for the entropies defined in (\ref{p-entropy}), $\psi''(1) = p(p-1)$. 
For the entropies defined in (\ref{p-entropy2}) %, $\psi''(1) = 0$ for $p > 2$, and is undefined for $1 < p < 2$. However, 
one has a substitute for the Csisz\'ar-Kullback inequality, namely the identity
$$\hat e_p(f_1|f_2) = \| f_1 - f_2\|^p_{L^p(f_2^{1-p})}\ .$$

To illustrate the standard entropy method on a very simple example, we first revisit the ODE (1.10) from \cite{ArMaToUn01} for the vector $f(t)=(f_1(t),\,f_2(t))^\top\in \R^2$:
\begin{align}\label{ODE1}
  \ddt{f}&= \lambda \A f,\quad t\ge 0\,,\\
  f(0) &= f^I \in\R^2\nonumber\,,
\end{align}
with the parameter $\lambda>0$, and the matrix $\A$ has BGK form:
\begin{equation}\label{BGK1}
  \A:=\left(\begin{array}{cc}
 -1 & 1 \\
  1 & -1
\end{array}\right)
=2\left[ 
\left(\begin{array}{c}
 \frac12   \\
  \frac12  
\end{array}\right)
\otimes \,(1,\;1) -
\left(\begin{array}{cc}
  1 & 0 \\
  0 & 1
\end{array}
\right)
\right]\,.
\end{equation}
This ODE can be seen as an $x$--homogeneous variant of \eqref{bgk2} with just two discrete velocities. In fact, on the right hand side of \eqref{BGK1}, the column vector $(\frac12,\;\frac12)^\top$ corresponds to the Maxwellian $M(v)$ in the BGK equation \eqref{bgk2}, and the row vector $(1,\;1)$ corresponds to the velocity integral. The symmetric matrix $\A$ has an eigenvalue 0 with corresponding eigenvector $f^\infty:=(\frac12,\;\frac12)^\top$ and an eigenvalue -2. Hence $\A$ is coercive on $\{f^\infty\}^\perp$. Since each column of $\A$ sums up to 0, the ``total mass'' of the system, i.e.\ $f_1(t)+f_2(t)$, stays constant in time. Hence, we shall assume w.l.o.g.\ that $f^I$ is normalized, i.e. $f^I_1+f^I_2=1$. Thus, as $t\to\infty$, $f(t)=f^\infty+(f^I-f^\infty)\ e^{-2\lambda t}$ converges to $f^\infty$ exponentially with rate $2\lambda$.
For $f^I_{1,2}\ge0$ we have $f_{1,2}(t)\ge0$.

In analogy to Definition \ref{def-rel-entr} we introduce for \eqref{ODE1} (with $n=2$) the relative entropy generated by $\psi$:
\begin{equation}\label{relative-entropy}
  e_\psi(f(t)|f^\infty):= \sum_{j=1}^n \psi\Big(\frac{f_j(t)}{f_j^\infty}\Big)f_j^\infty\,. 
\end{equation}
Its time derivative under the flow of \eqref{ODE1} reads
\begin{align}\label{Fisher-info}
  \ddt e_\psi(f(t)|f^\infty)&= -\lambda(f_1-f_2)\,\left[\psi'\Big(\frac{f_1(t)}{f_1^\infty}\Big)-\psi'\Big(\frac{f_2(t)}{f_2^\infty}\Big)\right]\\
  &   =:-I_\psi(f(t)|f^\infty)=-2\lambda\psi''(\zeta)(f_1-f_2)^2 \le0\,,\nonumber
\end{align}
%\end{eqnarray}
where $\zeta=\zeta(t)$ is an intermediate value between $2f_1(t)$ and $2f_2(t)$.
$I_\psi(f(t)|f^\infty)$ denotes the \emph{Fisher information} (of $f(t)$ w.r.t.\ $f^\infty$).

As pointed out in \cite{ArMaToUn01}, it is not obvious to bound this Fisher information from below directly by a multiple of the relative entropy (except for quadratic entropies). 
The goal of such an estimate would be to establish the exponential decay of the relative entropy. Hence, it is the essence of the entropy method to consider the entropy dissipation rate: Differentiating \eqref{Fisher-info} once more in time gives
\begin{align}\label{diss-rate}
  &R_\psi(f(t)|f^\infty) :=-\ddt I_\psi(f(t)|f^\infty)\\
  & \quad= 2\lambda I_\psi(f(t)|f^\infty) +\lambda^2 \big(f_1(t)-f_2(t)\big)^2\,\left[\psi''\Big(\frac{f_1(t)}{f_1^\infty}\Big)\frac1{f_1^\infty}
  +\psi''\Big(\frac{f_2(t)}{f_2^\infty}\Big)\frac1{f_2^\infty}\right]\,. \nonumber
\end{align}
Due to $\psi''\ge0$ the second term is nonnegative. Hence, 
$$
  -\ddt I_\psi(f(t)|f^\infty)\ge 2\lambda I_\psi(f(t)|f^\infty)\,.
$$
And this yielded in \cite{ArMaToUn01} the exponential decay of $I_\psi(f(t)|f^\infty)$ and of $e_\psi(f(t)|f^\infty)$ at the \emph{sub-optimal} rate $2\lambda$. But this procedure can be improved easily to give the following sharp result:
\begin{theorem}\label{th1}
Let the convex entropy generator $\psi$ satisfy either: $\psi''$ is convex on $J$; or $\psi'$ is concave on $(0,1)$ along with $\psi'$ is convex on $(1,\infty)$.
%Let the convex entropy generator $\psi$ either satisfy $\psi^{IV}\ge0$ on $J$ or $\psi'''\le0$ on $(0,1)$ along with $\psi'''\ge0$ on $(1,\infty)$ ($\psi'''$ does not have to be continuous at $\sigma=1$).
Then the solution to \eqref{ODE1} satisfies
\begin{align}\label{I-decay}
  I_\psi(f(t)|f^\infty) \le e^{-4\lambda t}\,I_\psi(f^I|f^\infty)\,,\quad t\ge0\,,\\
  e_\psi(f(t)|f^\infty) \le e^{-4\lambda t}\,e_\psi(f^I|f^\infty)\,,\quad t\ge0\,.\label{e-decay}
\end{align}
\end{theorem}
\begin{proof}
\underline{Case 1: $\psi''$ convex on $J$}\\
%\underline{Case 1: $\psi^{IV}\ge0$}\\
%Since $\psi''$ is convex we
We have for $0\le s\le1$:
$$
  s\psi''(\sigma_2)+(1-s)\psi''(\sigma_1) \ge \psi''\big(s \sigma_2+(1-s)\sigma_1\big) \,.
$$
Integrating this inequality over $s\in[0,1]$ yields $\forall\, \sigma_1\ne\sigma_2\in J$:
\begin{equation}\label{psi-cond}
  \frac{\psi''(\sigma_1)+\psi''(\sigma_2)}{2} \ge
  \kappa\frac{\int_{\sigma_1}^{\sigma_2} \psi''(\sigma) \d[\sigma]  }{\sigma_2-\sigma_1}
  =\kappa\frac{\psi'(\sigma_2)-\psi'(\sigma_1)}{\sigma_2-\sigma_1} \,,
\end{equation}
where $\kappa$ is introduced only for later reference. Here we set $\kappa=1$.

We now recall that $f_1^\infty=f_2^\infty$. Hence, \eqref{diss-rate} and \eqref{psi-cond} give
\begin{equation}\label{I-ineq}
  \ddt I_\psi(f(t)|f^\infty)\le -4\lambda I_\psi(f(t)|f^\infty)\,,
\end{equation}
and \eqref{I-decay} follows. As usual in the entropy method, one next integrates \eqref{I-ineq} in time (from $t$ to $\infty$) to obtain
%\begin{equation}\label{I-ineq}
$$
  \ddt e_\psi(f(t)|f^\infty)\le -4\lambda e_\psi(f(t)|f^\infty)\,,
%\end{equation}
$$
and this finishes the proof for the case $\psi''$ convex.\\

\underline{Case 2: $\psi'$ concave on $(0,1)$ along with $\psi'$ convex on $(1,\infty)$}\\
%\underline{Case 2: $\psi'''\le0$ on $(0,1)$ along with $\psi'''\ge0$ on $(1,\infty)$}\\
We may assume without loss of generality that $f_1 > f_2$. Then $f_1/f_1^\infty > 1> f_2/f_2^\infty$, and by the tangent line inequality for the concave function $\psi'\big|_{(0,1)}$~,
$$0 = \psi'(1) \leq  \psi'\left(\frac{f_2}{f_2^\infty}\right)  +  \psi''\left(\frac{f_2}{f_2^\infty}\right) \left(\frac{f_2^\infty - f_2}{f_2^\infty} \right)\,.$$
Likewise, using the tangent line inequality for the convex function $\psi'\big|_{(1,\infty)}$~,
%version of the above the tangent line inequality with the derivative evaluated at the other end of the interval,
$$\psi'\left(\frac{f_1}{f_1^\infty}\right) \leq \psi'(1) + \psi''\left(\frac{f_1}{f_1^\infty}\right) \left(\frac{f_1- f_1^\infty}{f_1^\infty} \right) =
\psi''\left(\frac{f_1}{f_1^\infty}\right) \left(\frac{f_1- f_1^\infty}{f_1^\infty} \right)\ . $$
Altogether we have
\begin{equation}\label{capri}
\psi''\left(\frac{f_1}{f_1^\infty}\right) \left(\frac{f_1- f_1^\infty}{f_1^\infty} \right) \geq \psi'\left(\frac{f_1}{f_1^\infty}\right) \quad{\rm and}\quad 
\psi''\left(\frac{f_2}{f_2^\infty}\right) \left(\frac{f_2^\infty - f_2}{f_2^\infty} \right) \geq   -\psi'\left(\frac{f_2}{f_2^\infty}\right)\ .
\end{equation}
Now continuing to assume that $f_1 > f_2$, and using the fact that $f_1^\infty = f_2^\infty$ so that $f_1 - f_2 = 2(f_1 - f_1^\infty) = 2 (f_2^\infty - f_2)$, 
\begin{eqnarray*}
&&\big(f_1-f_2\big)\,\left[\psi''\Big(\frac{f_1}{f_1^\infty}\Big)\frac1{f_1^\infty}
  +\psi''\Big(\frac{f_2}{f_2^\infty}\Big)\frac1{f_2^\infty}\right] \\
  &&\qquad\qquad = 2\big(f_1-f_1^\infty\big)  \psi''\Big(\frac{f_1}{f_1^\infty}\Big)\frac1{f_1^\infty} 
+ 2(f_2^\infty - f_2) \psi''\Big(\frac{f_2}{f_2^\infty}\Big)\frac1{f_2^\infty}\nonumber\\
&&\qquad\qquad\geq  2\left[\psi'\left(\frac{f_1}{f_1^\infty}\right)  -  \psi'\left(\frac{f_2}{f_2^\infty}\right)\right]\nonumber\,.
\end{eqnarray*}
Therefore, 
 $$\lambda^2 \big(f_1-f_2\big)^2\,\left[\psi''\Big(\frac{f_1}{f_1^\infty}\Big)\frac1{f_1^\infty}
  +\psi''\Big(\frac{f_2}{f_2^\infty}\Big)\frac1{f_2^\infty}\right] \geq 2\lambda I_\psi(f|f^\infty)\ .$$
  Again from (\ref{diss-rate}) we obtain (\ref{I-ineq}).
\smartqed
\qed
\end{proof}

\noindent
\underline{Remark:}
\begin{enumerate}
\item Concerning the logarithmic and power law entropies from \eqref{p-entropy} one easily verifies: $\psi_p$ satisfies the condition $\psi^{IV}\ge0$ on $J$ (or the inequality \eqref{psi-cond}) exactly for $p\in[1,2]\cup[3,\infty)$.
\item For $\psi_p$ with $p\in (2,3)$, inequality \eqref{psi-cond} holds with $\kappa=\frac{p-1}{2}$ (but not for any larger constant $\kappa$). This follows from $g_p(z):=z^{p-2}+1-\frac{z^{p-1}-1}{z-1}>0$ on $\R^+$ and $g_p(0)=0$, which can be verified by elementary computations. 
Hence, for $p\in(2,3)$, the entropy method yields exponential decay of $e_p(f(t)|f^\infty)$ with the reduced rate $2(\kappa+1)\lambda=(p+1)\lambda$: 
$$
  e_p(f(t)|f^\infty) \le   e^{-(p+1)\lambda t}e_{p}(f^I|f^\infty),\,\quad t\ge 0.
$$
But the decay estimates \eqref{e-decay}, \eqref{I-decay} are in general false for $p\in (2,3)$.\\

In an alternative approach, one can verify for $2<p<3$ the estimates
$$
  \psi_p(\sigma)\le\psi_3(\sigma),\quad\forall\,\sigma\ge0;
  \qquad \psi_3(\sigma)\le C_p \psi_p(\sigma),\quad\forall\,0\le\sigma\le2,
$$
where $[0,2]$ is the maximum range of values for $\frac{f_1}{f_1^\infty}$ and $\frac{f_2}{f_2^\infty}$.
Here the constant is $C_p=\frac{\psi_3(2)}{\psi_p(2)}=\frac{4}{2^p-1-p}$.
With \eqref{e-decay} this implies 
$$
  e_p(f(t)|f^\infty) \le e^{-4\lambda t}e_3(f^I|f^\infty) 
  \le C_p  e^{-4\lambda t}e_p(f^I|f^\infty),\,\quad t\ge 0.
$$
Hence, the entropies $e_p,\,p\in(2,3)$ still decay with the optimal rate $4\lambda$, but at the price of the multiplicative constant $C_p>1$.
\item The relative entropies $\hat e_p,\;p\ge2$ from \eqref{p-entropy2} satisfy the second set of assumptions in Theorem \ref{th1}. Note that $\psi'''$ does not have to be continuous at $\sigma=1$.
\end{enumerate}

%%%%%%%%%%%%%%%%%%%%%%%%%%%%%%%%%%%%%%%%%%%%%%%%%%%%%%%%%%%%%%%%%%%%%%%%%%%%%%%%%%%%%%%%%%5

\subsection{Multi-velocity BGK models}
Now, we consider discrete space-homogeneous BGK models in $\R^n$:
The evolution of a vector $f(t)=(f_1(t),\,f_2(t),\,\ldots,\,f_n(t))^\top\in \R^n$
 is governed by 
\begin{equation} \label{ODE1:general}
 \begin{cases}
  \ddt{f}= 2\lambda \A f,\quad t\ge 0\,,\\
  f(0)   = f^I \in\R^n\,,
\end{cases} \end{equation}
for some $\lambda>0$ and a matrix $\A\in\R^{n\times n}$ in BGK form
\begin{equation} \label{matrix:A:n}
 \A = \begin{pmatrix} \rho_1 \\ \vdots \\ \rho_n \end{pmatrix} \otimes (1,\,\ldots,\,1) - \II
%     = \begin{pmatrix}
%         \rho_1 -1 & \rho_1 & \ldots & \ldots & \rho_1 \\
%         \rho_2 & \rho_2 -1 & \rho_2 & \ldots & \rho_2 \\
%         \vdots & \vdots & \vdots & \vdots & \vdots \\
%         \rho_n & \ldots & \ldots & \rho_n & \rho_n -1 
%       \end{pmatrix}
\end{equation}
with $\rho=(\rho_1,\ldots,\rho_n)^\top \in(0,1)^n$ such that $\sum_{j=1}^n \rho_j = 1$.

Such a matrix~$\A$ has a simple eigenvalue $0$ with left eigenvector $l_1=(1,\ldots,1)$ and right eigenvector $r_1=\rho$,
 and an eigenvalue $-1$ with geometric multiplicity $n-1$.
Since each column of $\A$ sums up to 0, the ``total mass'' of system~\eqref{ODE1:general} %, i.e.\ $\sum_{j=1}^n f_j(t)$,
 stays constant in time, i.e. $\sum_{j=1}^n f_j(t) =\sum_{j=1}^n f_j^I$. 

Matrix $\A = (a_{jk})_{j,k=1,\ldots,n}$ 
 has only non-negative off-diagonal coefficients $a_{jk}$ $(j\ne k)$;
 such matrices are called \emph{essentially non-negative} or Metzler matrices~\cite{Se81}.
An essentially non-negative matrix $\A$ induces via~\eqref{ODE1:general}
 a semi-flow which preserves non-negativity of its initial datum $f^I$,
 i.e. $f^I_j\ge0$ for all $j=1,\ldots,n$, implies $f_j(t)\ge0$ for all $t\geq 0$.
\noindent

\underline{Remark:}
An essentially non-negative matrix is called \emph{$Q$-matrix} (or $W$-matrix in \cite{vK07})
 if it has an eigenvalue $0$ with right eigenvector $(1,\ldots,1)^\top$.
$Q$-matrices are the infinitesimal generators of continuous-time Markov processes with finite state space~\cite{No97}. 

In the following, we consider normalized positive initial data $f^I$, i.e. $\sum_{j=1}^n f^I_j=1$,
 such that the solution $f$ of~\eqref{ODE1:general} is positive and normalized for all $t\geq 0$.
Thus, as $t\to\infty$, $f(t)=f^\infty+(f^I-f^\infty)\ e^{-2\lambda t}$ converges to the normalized steady state $f^\infty:=\rho$ exponentially with rate $2\lambda$.

The study of the long-time behavior of solutions $f$ to \eqref{ODE1:general} is a classical topic,
 an approach via entropy methods can be found in~\cite{vK07,Pe07}.
Note that Perthame~\cite[\S 6.3]{Pe07} considers essentially positive matrices (i.e. off-diagonal elements are positive)
 to simplify the presentation.
However, the results generalize to irreducible $Q$-matrices,
 since only the non-negativity of off-diagonal elements is used,
 see also \cite[Remark 6.2]{Pe07}.
While \cite[Proposition 6.5]{Pe07} establishes only exponential decay in entropy,
% we aim at quantitative estimates on the decay rate.
 we aim at the optimal decay rate in the entropy approach.

We consider the time derivative of the relative entropy~\eqref{relative-entropy} under the flow of \eqref{ODE1:general}
\begin{equation} %\label{entropy-dissipation}
  \ddt e_\psi(f(t)|f^\infty)
%    = \sum_{j=1}^n \psi'\Big(\tfrac{f_j(t)}{f_j^\infty}\Big) \ddt f_j(t)
    = \sum_{j=1}^n \psi'\Big(\tfrac{f_j(t)}{f_j^\infty}\Big) 2\lambda\ (f_j^\infty -f_j(t))
    =: -I_\psi(f(t)|f^\infty)
    \leq 0
\end{equation}
 which is non-positive due to the properties~\eqref{entropy-generator} of an entropy generator
 ($\psi'$ is an increasing function with $\psi'(1)=0$).
%Hence, the Fisher information $I_\psi(f(t)|f^\infty) :=-\ddt e_\psi(f(t)|f^\infty)\geq 0$ is non-negative.
Next, we compute the second order derivative of $e_\psi(f(t)|f^\infty)$ w.r.t. time: 
\begin{align*} %\label{entropy-dissipation}
R_\psi(f(t)|f^\infty) :&= -\ddt I_\psi(f(t)|f^\infty) 
     = \ddt \sum_{j=1}^n \psi'\Big(\tfrac{f_j(t)}{f_j^\infty}\Big) \ddt f_j \\
    &= \sum_{j=1}^n \psi'\Big(\tfrac{f_j(t)}{f_j^\infty}\Big) \ddtddt f_j
        + \sum_{j=1}^n \psi''\Big(\tfrac{f_j(t)}{f_j^\infty}\Big) \tfrac1{f_j^\infty} \Big(\ddt f_j\Big)^2 \\
    &= 2\lambda I_\psi(f(t)|f^\infty) + \sum_{j=1}^n \psi''\Big(\tfrac{f_j(t)}{f_j^\infty}\Big) \tfrac1{f_j^\infty} \Big(\ddt f_j\Big)^2  
     \geq 2\lambda I_\psi(f(t)|f^\infty) \,,
\end{align*}
 since $\A^2=-\A$ and $\psi''\geq 0$.
% hence $\ddtddt f = - 2\lambda \ddt f$,
% and $\sum_{j=1}^2 \psi''\Big(\tfrac{f_j(t)}{f_j^\infty}\Big) \tfrac1{f_j^\infty} \Big(\ddt f_j\Big)^2\geq 0$.
This yields the non-optimal entropy dissipation rate $2\lambda$. 
To obtain a better entropy dissipation rate,
 we want to estimate the neglected term via
 \begin{equation} \label{est:neglected-term}
   \sum_{j=1}^n \psi''\Big(\tfrac{f_j(t)}{f_j^\infty}\Big) \tfrac1{f_j^\infty} \Big(\ddt f_j\Big)^2
     \geq \mu I_\psi(f(t)|f^\infty) \geq 0 
 \end{equation}
 for some $\mu>0$.
\begin{theorem}
Let $\rho=(\rho_1,\ldots,\rho_n)^\top \in(0,1)^n$ such that $\sum_{j=1}^n \rho_j = 1$ 
 and let the convex entropy generator $\psi\in C^2(J)$ satisfy for some $\mu>0$ and all $u=(u_1,\ldots,u_n)^\top \in[0,1]^n$ with $\sum_{j=1}^n u_j=1$: 
\begin{equation} \label{admissible-entropies:n}
%   \sum_{j=1}^n \psi''\Big(\tfrac{u_j}{\rho_j}\Big) \tfrac1{\rho_j} \Big(2\lambda(\rho_j - u_j)\Big)^2
%     \geq \mu \sum_{j=1}^n \psi'\Big(\tfrac{u_j}{\rho_j}\Big) 2\lambda(u_j - \rho_j) \geq 0.
  \sum_{j=1}^n \psi''\Big(\tfrac{u_j}{\rho_j}\Big) \tfrac1{\rho_j} (\rho_j - u_j)^2
    \geq \frac{\mu}{2\lambda}\ \sum_{j=1}^n \psi'\Big(\tfrac{u_j}{\rho_j}\Big) (u_j - \rho_j) . % \geq 0.
\end{equation}
Then, for all non-negative normalized initial data $f^I$, the solution~$f$ to \eqref{ODE1:general} satisfies
 \begin{align} 
  I_\psi(f(t)|f^\infty) \le e^{-(2\lambda+\mu) t}\,I_\psi(f^I|f^\infty)\,,\quad t\ge0\,,\label{I-decay:general:n}\\
  e_\psi(f(t)|f^\infty) \le e^{-(2\lambda+\mu) t}\,e_\psi(f^I|f^\infty)\,,\quad t\ge0\,.\label{e-decay:general:n}
 \end{align}
% with $f^\infty=\rho$.
% with $\rho=2\lambda+\mu>0$.
\end{theorem}
% \noindent
% \underline{Remark:}
%  Every convex entropy generator $\psi\in C^2(J)$ satisfies \eqref{admissible-entropies:n} at least for $\mu=0$.
%  In this case, we recover the non-optimal entropy dissipation rate $2\lambda$. 
% 
\begin{proof}
%Due to the assumptions on $f^I$,
The solution $f$ to \eqref{ODE1:general} is positive and normalized for all $t>0$.
Under Assumption~\eqref{admissible-entropies:n} on $\psi$,
 we obtain the estimates~\eqref{est:neglected-term}, and 
 \begin{equation}\label{I-ineq:general}
   \ddt I_\psi(f(t)|f^\infty)\le -(2\lambda +\mu)\ I_\psi(f(t)|f^\infty)\,,
 \end{equation}
 hence \eqref{I-decay:general:n} follows. 
Next, one integrates \eqref{I-ineq:general} in time (from $t$ to $\infty$) to obtain
 \[ \ddt e_\psi(f(t)|f^\infty)\le -(2\lambda +\mu)\ e_\psi(f(t)|f^\infty)\,, \]
 and this finishes the proof.
\smartqed \qed \end{proof}

For the quadratic entropy generator $\psi_2$ inequality \eqref{admissible-entropies:n} holds with $\mu=2\lambda$.
Thus we recover the optimal decay rate $4\lambda$ in \eqref{I-decay:general:n}--\eqref{e-decay:general:n}.
For the logarithmic entropy generator $\psi_1$ an estimate for $\mu$ in \eqref{admissible-entropies:n}
 has been given in \cite{DiSaCo96,BoTe06} as 
 \[ \frac{\mu}{2\lambda} \geq \sqrt{\rho_{min} (1-\rho_{min})} \qquad \text{with } \rho_{min} = \min_{j=1,\ldots,n} \rho_j . \]
Next, we consider entropy generators $\psi$ in the sense of Definition~\ref{def-rel-entr},
 such that $\psi'$ is concave on $(0,1)$ along with $\psi'$ convex on $(1,\infty)$.
Thus, for $f_1\geq f_1^\infty>0$ and $f_2^\infty\geq f_2>0$, the inequalities~\eqref{capri} continue to hold.
Distinguishing the cases $u_j<\rho_j$, $u_j>\rho_j$ and the trivial case $u_j=\rho_j$, we deduce for all $j=1,\ldots,n$,
 \begin{equation*} %\label{admissible-entropies:n}
  \psi''\Big(\tfrac{u_j}{\rho_j}\Big) \tfrac1{\rho_j} (\rho_j - u_j)^2
    \geq \psi'\Big(\tfrac{u_j}{\rho_j}\Big) (u_j - \rho_j) ,
 \end{equation*}
 hence \eqref{admissible-entropies:n} holds with $\mu=2\lambda$.
However, for the entropy generators $\hat \psi_p$ in \eqref{p-entropy2} with $p\geq 2$ 
 the optimal value is $\mu=(p-1)2\lambda$.

%For a given entropy generator $\psi$, the best constant $\mu$ in \eqref{admissible-entropies:n} can be calculated explicitly.
In the following,
 we restrict ourselves to $n=2$ and determine the best constant for some polynomial entropy generators:
\begin{lemma}
Let $\rho_1, \rho_2\in(0,1)$ with $\rho_1+\rho_2=1$.
The entropy generator $\psi(\sigma)$ satisfies
 condition~\eqref{admissible-entropies:n} with
 \[ 1\geq \frac{\mu}{2\lambda} =\begin{cases}
     1 &\text{for}\quad \psi(\sigma)=\psi_2(\sigma) , \\
     2 \min \{\rho_1,\ \rho_2\} &\text{for}\quad \psi(\sigma)=\psi_3(\sigma) , \\
     2-2\sqrt{1-3 \rho_2 \ (1-\rho_2)}>0 &\text{for}\quad \psi(\sigma)=\psi_4(\sigma) . 
                   \end{cases}
 \]
\end{lemma}
\begin{proof}
For $n=2$, 
 the assumptions on $\rho$ and $u$ in \eqref{admissible-entropies:n} imply 
 \begin{equation*}
  -(\rho_2 -u_2) = \rho_1 -u_1 = \rho_1 (u_1 +u_2) -u_1 %= \rho_1 u_2 - (1-\rho_1)u_1
   = \rho_1 u_2 - \rho_2 u_1 = \rho_1 \rho_2 (\tfrac{u_2}{\rho_2} -\tfrac{u_1}{\rho_1}).
 \end{equation*}
Thus condition \eqref{admissible-entropies:n} is equivalent to
\begin{equation*}
  \sum_{j=1}^2 \psi''\Big(\tfrac{u_j}{\rho_j}\Big) \tfrac1{\rho_j} \rho_1^2 \rho_2^2 (\tfrac{u_2}{\rho_2} -\tfrac{u_1}{\rho_1})^2
    \geq \frac{\mu}{2\lambda}\ \sum_{j=1}^2 \psi'\Big(\tfrac{u_j}{\rho_j}\Big) (-1)^j \rho_1 \rho_2 (\tfrac{u_2}{\rho_2} -\tfrac{u_1}{\rho_1}) \geq 0.
\end{equation*}
Setting $v_1:=u_1/\rho_1$ and $v_2:=u_2/\rho_2$, we deduce for $\psi_p(\sigma)$, $p>1$,
 \[ (p-1) \big[ v_1^{p-2} \rho_2 + v_2^{p-2} \rho_1 \big] (v_1 -v_2)^2 
      \geq \frac{\mu}{2\lambda} \big[ v_1^{p-1} -v_2^{p-1}\big] (v_1 -v_2)
    \quad \forall v_1, v_2\geq 0 .
 \]
Moreover, for $v_2>0$, dividing by $v_2^p$ and defining $z:=v_1/v_2$, we obtain
 \[ (p-1) \big[ z^{p-2} \rho_2 + \rho_1 \big] (z -1)^2 
      \geq \frac{\mu}{2\lambda} \big[ z^{p-1} -1\big] (z -1) \qquad \forall z\geq 0 .
 \]

We show the statement for the quartic entropy generator $\psi_4(\sigma)$, 
 the (simpler) proof for quadratic and cubic entropy generators is omitted.
For~$\psi_4$, condition~\eqref{admissible-entropies:n} is equivalent to
 \[ g(z) := z^2 (3 \rho_2 -\tilde{\mu}) -\tilde{\mu} z + 3 \rho_1 -\tilde{\mu} \geq 0 \qquad \forall z\geq 0 \]
 with $\tilde{\mu}:=\mu/(2\lambda)$.
Evaluating $g(z)$ at $z=0$ and taking the limit $z\to\infty$,
 we deduce the necessary conditions $3 \rho_1 \geq \tilde{\mu}$ and $3 \rho_2 > \tilde{\mu}$, respectively.
The minimum of $g(z)$ on $z\in(0,\infty)$
% \[ g'(z_0) = 2 z_0 (3 \rho_2 -\tilde{\mu}) -\tilde{\mu} = 0 \quad \text{such that} \quad z_0 = \frac{\tilde{\mu}}{2 (3 \rho_2 -\tilde{\mu})}. %\]
%Finally, the minimum $g(z_0)$
%%  \[ g(z_0) = -\frac{\tilde{\mu}^2}{4(3 \rho_2-\tilde{\mu})} + 3 \rho_1 -\tilde{\mu} \]
% of $g(z)$ on $(0,\infty)$
 is zero,
 iff $\tilde{\mu}$ solves $\tilde{\mu}^2 -4 (3 \rho_1 -\tilde{\mu})\ (3 \rho_2-\tilde{\mu})=0$.
This quadratic polynomial has a simple positive zero given by $\tilde{\mu}_0=2-2\sqrt{1-3 \rho_2 \ (1-\rho_2)}>0$,
 since $\rho_1+\rho_2=1$.

The expression $\tilde{\mu}_0=2-2\sqrt{1-3 \rho_2 \ (1-\rho_2)}>0$ attains its maximum $1$ for $\rho_2\in(0,1)$ at $\rho_2=1/2$.
\smartqed \qed \end{proof}

\noindent
\underline{Remark:}
 The quadratic entropy $\psi_2(\sigma)$ satisfies
  Assumption~\eqref{admissible-entropies:n} with $\mu=2\lambda$ for all $f^\infty_1,f^\infty_2\in(0,1)$.
 The cubic entropy $\psi_3(\sigma)$ and the quartic entropy $\psi_4(\sigma)$
  satisfy~\eqref{admissible-entropies:n} with $\mu=2\lambda$ only for $f^\infty_1 = f^\infty_2 = \tfrac12$.

%%%%%%%%%%%%%%%%%%%%%%%%%%%%%%%%%%%%%%%%%%%%%%%%%%%%%%%%%%%%%%%%%%%%%%%%%%%%%%%%%

\section{A discrete hypocoercive BGK model}
\label{sec:3}

In this section we consider an example for a discrete version (both in $x$ and $v$) of \eqref{bgk}. More precisely, we consider the evolution of a vector 
$f(t)=\big(f_j(t);\,j=1,...,4\big)^\top\in \R^4$, 
%$f(t)=(f_1(t),\,f_2(t),\,f_3(t),\,f_4(t))^\top\in \R^4$, 
where its four components may correspond to the following points in the $x-v$--phase space: $(1,1)$, $(1,-1)$, $(-1,-1)$, $(-1,1)$, in this order. Its evolution is given by
\begin{align}\label{ODE2}
  \ddt{f}&= (\A+\B) f,\quad t\ge 0\,,\\
  f(0) &= f^I \in\R^4\nonumber\,.
\end{align}
Similarly to \eqref{BGK1}, the matrix $\A$ has BGK form:
\begin{equation}\label{BGK2}
  \A:=\frac12 \left(\begin{array}{cccc}
 -1 &  1 &  0 &  0 \\
  1 & -1 &  0 &  0 \\
  0 &  0 & -1 &  1 \\
  0 &  0 &  1 & -1 
\end{array}\right)
=\frac12 \left(\begin{array}{cc}
  1 & 1 \\
  1 & 1
\end{array}
\right)
\otimes \, \left(\begin{array}{cc}
  1 & 0 \\
  0 & 1
\end{array}
\right) - \II \,,
\end{equation}
where the first summand on the r.h.s.\ is the projection onto the kernel of $\A$, 
$$
  \ker \A = \spn[(1\,1\,0\,0)^\top,(0\,0\,1\,1)^\top]\,.
$$

In \eqref{ODE2}, the matrix $\B$ is skew-symmetric and reads
\begin{equation}\label{B-matrix}
  \B:= \left(\begin{array}{cccc}
  0 & -1 &  0 &  1 \\
  1 &  0 & -1 &  0 \\
  0 &  1 &  0 & -1 \\
 -1 &  0 &  1 &  0 
\end{array}\right)
= \left(\begin{array}{cc}
  0 & -1 \\
  1 &  0
\end{array}
\right)
\otimes \, \left(\begin{array}{cc}
  1 & -1 \\
 -1 &  1
\end{array}
\right) \,.
\end{equation}
$\B$ corresponds to a discretization of the transport operator in \eqref{bgk} by symmetric finite differences. We remark that  \eqref{ODE2} does not preserve positivity but, as we shall show, the hypocoercivity of \eqref{bgk}. Motivated by the theory of hyperbolic systems, one may also replace the transport operator by an upwind discretization with a then non-symmetric matrix $\tilde\B$. Then, the resulting system would preserve positivity. But it would be coercive rather than hypocoercive. Here we opt to discuss the situation with $\B$ given in \eqref{B-matrix}.

The spectrum of $\A+\B$ is given by $0,\, -\frac12 \pm  \frac{\sqrt{15}}{2}\,i,\, -1$. 
The unique, (in the 1-norm) normalized steady state of \eqref{ODE2} is given by $f^\infty=w_1=\frac14 \,(1 1 1 1)^\top$, which spans the kernel of $\A+\B$. Eigenvectors of the non-trivial eigenvalues are given by $w_{2,3}:=$\linebreak 
$(\sqrt5,\,\pm\sqrt3 i,\,-\sqrt5,\,\mp\sqrt3 i)^\top$ and $w_4:=(1,\,-1,\,1,\,-1)^\top$, and all three of them have mass 0. 
This shows that $\frac12$ is the sharp decay rate of any (normalized) $f(t)$ towards $f^\infty$.
But this ``spectral gap'' of size $\frac12$ disappears in the symmetric part of the matrix:
$\sigma\big(\A+\frac{\B+\B^\top}{2}\big)=\{ 0,\, 0,\, -1,\, -1\}$. 
Hence, the matrix $\A+\B$ is only \emph{hypocoercive} on $\{f^\infty\}^\perp$ (as defined by C.\ Villani, see \S3.2 of \cite{ViH06}).
But using an appropriate similarity transformation of $\A+\B$ one can again recover the sharp decay rate of the hypocoercive BGK-model \eqref{ODE2} via energy or entropy methods. \\

In particular, we shall use Lyapunov's direct method --see Lemma~\ref{lemma:2norm:decay} in the following subsection-- to prove decay to equilibrium for normalized solutions:
If $f^I$ is normalized, % in the 1-norm,
 then the solution to \eqref{ODE2} satisfies (for any norm on $\R^4$)
 \[
  \|f(t)-f^\infty\| \le c\,e^{-t/2} \|f^I-f^\infty\|\,,
  \quad t\ge0\,,
 \]
 with some generic constant $c\ge1$.

%%%%%%%%%%%%%%%%%%%%%%%%%%%%%%%%%%%%%%%%%%%%%%%%%%%%%%%%%%%%%%%%%%%%%%%%%%%%%%%%%%555

\subsection{Lyapunov's direct method}

We consider an ODE for a vector $f(t)=(f_1(t),\,f_2(t),\ldots,f_n(t))^\top\in \R^n$:
 \begin{equation} \label{ODE:general:n}
  \begin{cases}
   \ddt{f}= \A f,\quad t\ge 0\,,\\
   f(0)   = f^I \in\R^n\,,
  \end{cases}
 \end{equation}
 for some real (typically non-symmetric) matrix $\A\in\R^{n\times n}$. 
The origin $0$ is a steady state of \eqref{ODE:general:n}.
The stability of the trivial solution $f^0(t)\equiv 0$ is determined by the eigenvalues of matrix $\A$:
\begin{theorem} \label{thm:f0:stability}
 Let $\A\in\R^{n\times n}$ and let $\lambda_j$ ($j=1,\ldots,n)$ denote the eigenvalues of $\A$
  (counted with their multiplicity).
 \begin{enumerate}[label=(S\arabic*)]
  \item The equilibrium $f^0$ of \eqref{ODE:general:n} is stable %uniformly stable 
    if and only (i) $\Re \lambda_j \leq 0$ for all $j=1,\ldots,n$;
    and (ii) all eigenvalues with $\Re \lambda_j=0$ are 
    non-defective\footnote{An eigenvalue is defective if its geometric multiplicity is strictly less than its algebraic multiplicity.}.
%     semi-simple, i.e. its algebraic and geometric multiplicity are equal.
  \item \label{f0:AS} The equilibrium $f^0$ of \eqref{ODE:general:n} is asymptotically stable %exponentially stable in the large 
    if and only if $\Re \lambda_j< 0$ for all $j=1,\ldots,n$.
  \item The equilibrium $f^0$ of \eqref{ODE:general:n} is unstable 
    in all other cases.
 \end{enumerate}
\end{theorem}
To study the stability for $f^0$ via Lyapunov's direct method,
 a first guess for a Lyapunov function $V(f)$ is the (squared) Euclidean norm $V(f)=\norm{f}_2^2$.  
The derivative of $V(f)$ along solutions $f(t)$ of \eqref{ODE:general:n} satisfies
 \begin{align*}
  \ddt V(f(t))
%    &= \ddt \ip{f(t)}{f(t)} = \ip{\ddt f(t)}{f(t)} + \ip{f(t)}{\ddt f(t)} \\
%    &= \ip{\A f(t)}{f(t)} + \ip{f(t)}{\A f(t)}
     = \ip{f(t)}{(\A^\top + \A)f(t)} \,.
 \end{align*}
Thus the derivative depends only on the symmetric part $\tfrac12 (\A^\top + \A)$ of a matrix $\A$.
Hence the choice $V(f)=\norm{f}_2^2$ is only suitable for symmetric matrices $\A$.

To study the stability of $f^0(t)\equiv 0$ w.r.t. \eqref{ODE:general:n} for a general $\A$,
 it is standard to consider the generalized (squared) norm
 \[ V(f) := \ip{f}{\P f} \quad \text{for some symmetric, positive definite matrix $\P\in\R^{n\times n}$.} \]
The derivative of $V(f)$ along solutions $f(t)$ of \eqref{ODE:general:n} satisfies
 \begin{align}
  \ddt V(f(t))
%    &= \ddt \ip{f(t)}{\P f(t)} = \ip{\ddt f(t)}{\P f(t)} + \ip{f(t)}{\P \ddt f(t)} \nonumber \\
    &= \ip{\A f(t)}{\P f(t)} + \ip{f(t)}{\P \A f(t)} = \ip{f(t)}{\RR f(t)} \,, \label{GDF:derivative}
%     = \ip{f(t)}{(\A^\top \P +\P \A)f(t)} \,. \label{GDF:derivative}
 \end{align}
 with matrix $\RR:= \A^\top \P +\P \A$.
Conclusions on the stability of $f^0$ are possible,
 depending on the (negative) definiteness of $\RR$, see e.g. \cite[Proposition 7.6.1]{MiHouLiu15}. 

To determine the decay rate of an asymptotically stable steady state,
 we shall use the following algebraic result. %(the complex analog of Lemma 4.3 in \cite{ArEr14} or Lemma 2.6 in \cite{AAS15}).
\begin{lemma}\label{Pdefinition}
For any fixed matrix $\CC\in\C^{n\times n}$,
 let $\mu:=\min\{\Re\{\lambda\}|\lambda$ is an eigenvalue of $\CC\}$. 
Let $\{\lambda_{j}|1\leq j\leq j_0\}$ be all the eigenvalues of $\CC$ with $\Re\{\lambda_j\}=\mu$, only counting their geometric multiplicity.

%  \begin{itemize}
%  \item[(i)]
 If all $\lambda_j$ ($j=1,\dots,j_0$) are non-defective, %\footnote{An eigenvalue is defective if its geometric multiplicity is strictly less than its algebraic multiplicity.},
  then there exists a Hermitian, positive definite matrix $\P\in\C^{n\times n}$ with
  \begin{align} \label{matrixestimate}
   \CC^*\P+\P \CC &\geq 2\mu \P\,,
  \end{align}
  where $\CC^*$ denotes the Hermitian transpose of $\CC$.
% \item[(ii)] If $\lambda_m$ is defective for at least one $m\in\{1,\dots,m_0\}$, then for any $\eps>0$ there exists a Hermitian, positive definite matrix $\P=\P(\eps)\in\C^{n\times n}$ with
% \begin{align}
% \label{degeneratematrixestimate} \CC^*\P+\P \CC &\geq 2(\mu-\eps) \P\,.
% \end{align}
%  \end{itemize}
 Moreover, (non-unique) matrices $\P$ satisfying~\eqref{matrixestimate} are given by
 \begin{align} \label{simpleP} 
  \P:= \sum\limits_{j=1}^n b_j \,w_j\otimes \overline{w_j}^\top\,,
 \end{align}
 where $w_j$ ($j=1,\dots,n$) denote the eigenvectors of $\CC^*$, and $b_j\in\R^+$ ($j=1,\dots,n$) are arbitrary weights.
\end{lemma}

\noindent
\underline{Remark:}
Lemma~\ref{Pdefinition} is the complex analog of \cite[Lemma 4.3]{ArEr14} or \cite[Lemma 2.6]{AAS15}.
In particular, if $\CC\in\R^{n\times n}$ is a real matrix, 
 then the inequality~\eqref{matrixestimate} of Lemma~\ref{Pdefinition} holds true for real, symmetric, positive definite matrices $\P\in\R^{n\times n}$.
Moreover, the case of defective eigenvalues is also treated in \cite{ArEr14,AAS15}.\\

If $\A\in\R^{n\times n}$ has only eigenvalues with negative real parts,
 then the origin is the unique and asymptotically stable steady state $f^0=0$ of $\eqref{ODE:general:n}$.
Due to Lemma~\ref{Pdefinition}, 
 there exists a symmetric, positive definite matrix $\P\in\R^{n\times n}$ such that $\A^\top \P +\P\A \leq -2\mu \P$
 where $\mu = \min |\Re \lambda_j|$.
Thus, the derivative of $V(f):= \ip{f}{\P f}$ along solutions of \eqref{ODE:general:n} satisfies
 \begin{align}
  \ddt V(f(t)) 
%     &= \ddt \ip{f(t)}{\P f(t)} = \ip{\ddt f(t)}{\P f(t)} + \ip{f(t)}{\P \ddt f(t)} \nonumber \\
%     &= \ip{\A f(t)}{\P f(t)} + \ip{f(t)}{\P \A f(t)} = \ip{f(t)}{(\A^\top \P +\P \A)f(t)} \,. \label{GDF:derivative}
    &\leq -2\mu V(f(t)) \qquad \text{with} \quad \mu = \min |\Re \lambda_j|,
 \end{align} 
 which implies $V(f(t))\leq e^{-2\mu t} V(f^I)$ and $\norm{f(t)}^2\leq c e^{-2\mu t} \norm{f^I}^2$ for some $c\ge1$ 
 by equivalence of norms on $\R^n$.

In contrast, we consider next matrices $\A\in\R^{n\times n}$
having only eigenvalues %$\lambda_j$ ($j=1,\ldots,n$) 
with non-positive real part.
More precisely, let $\A$ satisfy
 \begin{enumerate}[label=(A\arabic*)]
  \item \label{MatrixA:n:ConsMass}
   $\A$ has a simple eigenvalue $\lambda_1=0$ with left eigenvector $w_1^\top\in \R^n$ and right eigenvector $v_1\in\R^n$;
 \item \label{MatrixA:n:hypocoercive}
   the other eigenvalues $\lambda_j$ ($j=2,\ldots,n$) of $\A$ have negative real part.
 \end{enumerate}
% \begin{enumerate}[leftmargin=2em, label=A\arabic*)]
%  \item \label{MatrixA:n:NegSemiDef}
%    $\A$ is negative semi-definite, i.e. $\ip{v}{\A v} \leq 0$ for all $v\in\R^n$;
%  \item \label{MatrixA:n:ConsMass}
%    $\A$ has a simple eigenvalue $0$ with left eigenvector $(1\,\ldots\, 1)\in \R^{1\times n}$ and positive right eigenvector $f^\infty\in\R^n$;
%  \item \label{MatrixA:n:Symmetry}
%    $\A$ is symmetric with respect to the inner product
%    $\ip{v}{w}_{1/f^\infty} = \sum_{i=1}^n \tfrac{v_i \, w_i}{f^\infty_i}$. % = \ip{v}{\diag\big(\tfrac1{f^\infty_1},\tfrac1{f^\infty_2}\big) w}$. 
% \end{enumerate}
% Defining $\DD:=\diag\big( f^\infty_1,\ldots, f^\infty_n \big)$,
%  the inner product in \ref{MatrixA:Symmetry} satisfies $\ip{v}{w}_{1/f^\infty} = \ip{v}{\DD^{-1} w}$.
% Due to \ref{MatrixA:Symmetry}, the matrix $\DD^{-1}\A$ is symmetric (with respect to the inner product $\ip{\cdot}{\cdot}$).
Then, the space of steady states of~\eqref{ODE:general:n} consists of $\spn\{v_1\}$,
 and solutions to~\eqref{ODE:general:n} will typically not decay to $0$.
More precisely, if $f$ is a solution of ODE~\eqref{ODE:general:n} with initial datum $f^I$ satisfying $\ip{w_1}{f^I} = c$ for some $c\in\R$,
 then $\ip{w_1}{f(t)} = c$ for all $t\geq 0$.
Therefore we aim to prove the convergence of solutions $f(t)$ of \eqref{ODE:general:n} 
 for an initial datum $f^I$ (normalized in the sense of $\ip{w_1}{f^I}=1$)
 to the unique steady state $f^\infty\in\spn\{v_1\}$ (again normalized as $\ip{w_1}{f^\infty}=1$).

\begin{lemma} \label{lemma:2norm:decay}
 Let $\A\in\R^{n\times n}$ satisfy \ref{MatrixA:n:ConsMass}--\ref{MatrixA:n:hypocoercive} with non-defective eigenvalues $\lambda_j$ for $j=1,\ldots,n$.
%  let $\psi\in C^2(\R;\R)$ satisfy $\psi(1)=\psi'(1)=0$, as well as 
%   \begin{align}
%     \ip{\psi'(x)-\psi'(y)}{(-\A\DD) (x-y)} \geq 0 \quad \text{for all } x,y \,, \label{ineq:psi:PosSemiDef}
%    \intertext{and}
%     \ip{\psi'(x)-\psi'(y)}{(-\A^2\DD) (x-y)} \geq \ldots \quad \text{for all } x,y \,,
%   \end{align}
 If $f$ is a solution of \eqref{ODE:general:n} for some normalized initial datum $f^I$ (i.e. $\ip{w_1}{f^I}=1$),
  then 
  \begin{equation} \label{A:n:general:entropy:decay}
    \norm{f(t)-f^\infty} \leq c\ \norm{f^I-f^\infty} e^{-\lambda_* \, t} \,, 
    \quad t\ge0\,,
  \end{equation}
  where $\lambda_* := \min_{\lambda_j \ne 0} |\Re \lambda_j|$ and some constant $c\ge1$.
\end{lemma}
\begin{proof}
To present a unified approach for symmetric and non-symmetric matrices $\A$ satisfying~\ref{MatrixA:n:ConsMass}--\ref{MatrixA:n:hypocoercive},
 we consider again the ``distorted'' vector norm $\normP{f} := \sqrt{\ip{f}{\P f}}$, 
 and the relative entropy-type functional 
 \[
   E_{\psi_2}(f(t)|f^\infty) := \normP{f(t)-f^\infty}^2
%        = \sum_{j=1}^n f^\infty_j \psi_2'(\tfrac{f_j(t)}{f^\infty_j}) (\P (f-f^\infty))_j 
%        = \sum_{j=1}^n (f(t)-f^\infty)_j (\P (f-f^\infty))_j
 \]
 with some real, symmetric and positive definite matrix $\P$ to be determined.
Its derivative satisfies 
 \begin{align*}
   \ddt E_{\psi_2}(f(t)|f^\infty) 
%        &= \ddt \ip{f(t)-f^\infty}{\P (f(t)-f^\infty)} \\
%        &= \ip{\ddt (f(t)-f^\infty)}{\P (f(t)-f^\infty)} + \ip{f(t)-f^\infty}{\P \ddt (f(t)-f^\infty)}\\
%        &= \ip{\A (f(t)-f^\infty)}{\P (f(t)-f^\infty)} + \ip{f(t)-f^\infty}{\P \A (f(t)-f^\infty)}\\
       &= \ipBig{(f-f^\infty)}{(\A^\top \P + \P\A)(f-f^\infty)} \,. 
 \end{align*}
% Due to the fundamental theorem of linear algebra, 
Every matrix $\A\in\R^{n\times n}$ induces an orthogonal decomposition of $\R^n$ via 
 \[ \R^n = \ker(\A)\ \oplus\ \ran(\A^\top) = \ker(\A^\top)\ \oplus\ \ran(\A) . \]
Thus, there exists an orthogonal projection from $\R^n$ onto $\ran(\A)$,
 which is represented by a matrix $\P_1\in\R^{n\times n}$ with $\P_1^2 =\P_1$.
Due to assumption~\ref{MatrixA:n:ConsMass},
 matrix $\A^\top$ has a one-dimensional kernel which is spanned by $w_1$,
 hence $\P_1 w_1 = 0$. 
%  We choose $\P= \sum_{j=1}^n b_j \,w_j\otimes \overline{w_j}^\top$ as in~\eqref{simpleP},
%   where $w_1,\dots,w_n$ denote the eigenvectors of $\A^\top$
%   (i.e. $w_j^\top$ is the left eigenvector of $\A$ associated to the eigenvalue $\lambda_j$),
%   and $b_j\in\R^+$ ($j=1,\dots,n$) are arbitrary weights.
%  Moreover, $v_j$ are the right eigenvectors of $\A$ associated to the eigenvalue $\lambda_j$ ($j=1,\ldots,n$); and
%   \[ \ip{w_1}{v_j} \begin{cases} \ne 0 &\text{for } 1=j \,, \\ = 0 &\text{otherwise} \,. \end{cases} \]
% %  whereat $\lambda_1=0$ and all other eigenvalues have negative real part.
Since $w_1^\top$ is a left eigenvector of $\A$ for the eigenvalue $0$, 
 a solution $f$ of \eqref{ODE:general:n} for a normalized initial datum $f^I$ (i.e. $\ip{w_1}{f^I}=1$)
 is again normalized, i.e. $\ip{w_1}{f(t)}=1$ for all $t\geq 0$.
Thus, $\ip{w_1}{f(t)-f^\infty}\equiv 0$ iff $\ip{w_1}{f^I-f^\infty}=0$,
 which implies $f(t)-f^\infty \in \ran(\A)$ for all $t\geq 0$.
Moreover,
 \begin{align*}
   \ddt E_{\psi_2}(f(t)|f^\infty) 
%        &= \ddt \ip{f(t)-f^\infty}{\P (f(t)-f^\infty)} \\
%        &= \ip{\ddt (f(t)-f^\infty)}{\P (f(t)-f^\infty)} + \ip{f(t)-f^\infty}{\P \ddt (f(t)-f^\infty)}\\
%        &= \ip{\A (f(t)-f^\infty)}{\P (f(t)-f^\infty)} + \ip{f(t)-f^\infty}{\P \A (f(t)-f^\infty)}\\
%        &= \ipBig{\P_1(f-f^\infty)}{(\A^\top \P + \P\A)\P_1(f-f^\infty)} \,. 
       &= \ipBig{\P_1(f-f^\infty)}{\P_1^\top(\A^\top \P + \P\A)\P_1\ \P_1(f-f^\infty)} \,. 
 \end{align*}
In order to prove 
\begin{equation}\label{P1-ineq}
  \P_1^\top(\A^\top \P + \P\A)\P_1 \leq -2\lambda_* \P_1^\top \P \P_1 
\end{equation}
we consider the modified matrix $\tilde\A := \A -\lambda_* v_1\otimes w_1^\top \in\R^{n\times n}$.
Due to \ref{MatrixA:n:ConsMass}--\ref{MatrixA:n:hypocoercive} and the assumptions in Lemma~\ref{lemma:2norm:decay},
$\tilde\A$ has only non-defective eigenvalues with negative real part.
Due to Lemma~\ref{Pdefinition}, 
 there exists a real, symmetric, positive-definite matrix $\P$ such that 
 $\tilde{A}^\top \P + \P\tilde{A} \leq -2\lambda_* \P$.
This implies \eqref{P1-ineq}
since $\P_1^\top\big((v_1\otimes w_1^\top)^\top \P + \linebreak \P(v_1\otimes w_1^\top)\big)\P_1 = 0$.
Therefore we conclude
  \begin{equation} \label{Epsi2:decay}
    \ddt E_{\psi_2}(f(t)|f^\infty) 
    %\leq -2 \min_{\lambda_j\ne 0} |\Re \lambda_j| E_{\psi_2}(f(t)|f^\infty) 
    \leq -2 \lambda_* E_{\psi_2}(f(t)|f^\infty) \,,
  \end{equation}
% Due to Gronwall's lemma and $E_{\psi_2} \geq 0$, we obtain
and $E_{\psi_2}(f(t)|f^\infty) \leq E_{\psi_2}(f^I|f^\infty) e^{-2\lambda_* \, t}$ follows.
Moreover, $0\leq \lambda_{P,min} \II \leq\P \leq\lambda_{P,max} \II$,
 where $\lambda_{P,min}>0$ is the smallest eigenvalue
 and $\lambda_{P,max}>0$ is the biggest eigenvalue of $\P$.
Therefore, $\lambda_{P,min} \norm{f}_2^2 \leq \normP{f}^2\leq \lambda_{P,max} \norm{f}_2^2$
 and \eqref{A:n:general:entropy:decay} follows.
\smartqed \qed \end{proof}
\noindent
\underline{Remark:}
For a symmetric matrix $\A$, the choice $\P=\II$ is admissible 
and one recovers the optimal decay rate and constant $c=1$ in estimate~\eqref{A:n:general:entropy:decay}.\\

\noindent
\underline{Remark:}
Assume now that the matrix $\A$ from Lemma \ref{lemma:2norm:decay} satisfies also $\ker (\A)=\ker (\A^\top)$, which corresponds to \emph{detailed balance} for the steady state. Then, Lemma \ref{lemma:2norm:decay} allows for a simpler proof: Let $w_1=f^\infty\in\R^n$ be a normalized steady state. Then the orthogonal projector $w_1\otimes \overline{w_1}^\top$ commutes with both $\A$ and $\A^\top$. Let  $\P_1$ denote its complementary projection. Then $\ran(\P_1)$ is invariant under $e^{\A t}$, and \eqref{P1-ineq} with $\P$ from \eqref{simpleP} follows from Lemma \ref{Pdefinition} applied to $\A$ restricted to $\ran(\P_1$).

\section{Space-inhomogeneous BGK models}
\label{sec:4}

In this section we study the large-time behavior of the BGK equation \eqref{bgk2} on $L^2(\mathbf{T}^1\times\R;M_T^{-1}(v)\d[v])$ with periodic boundary conditions in $x$. We start with the $x$--Fourier series of $f$:
\begin{equation}\label{Fourier}
  f(x,v,t) = \sum_{k\in\Z} f_k(v,t)\,e^{ikx}\,,
\end{equation}
and obtain the following evolution equation for the spatial modes $f_k,\,k\in\Z$:
\begin{equation}\label{bgk-modes}
  \partial_t f_k +ikvf_k = \Q f_k = M_T(v)\int_\R f_k(v,t)\,{\rm d} v-f_k(v,t)\,,\quad k\in\Z;\;t\ge0\,.
\end{equation}
Since the BGK operator $\Q$ projects onto the centered Maxwellian at temperature $T$, it is natural to consider \eqref{bgk-modes} in the basis spanned by the Hermite functions (in $v$). This is natural for the following reason:

The Hermite polynomials (for temperature $T$) are the system of orthonormal 
polynomials that one obtains by applying the Gram-Schmidt orthonormalization
procedure to the sequence of monomials $\{v^\ell\}$  in $L^2(M_T)$; 
let $P_\ell(v)$ denote the $\ell$th Hermite polynomial.  The {\em Hermite functions} themselves are the  
functions of the form $\tilde g_\ell(v) = P_\ell(v)M_T(v)$, and evidently these are orthonormal in $L^2(M_T^{-1})$.
This is the space in which we work. 

The key fact concerning the Hermite functions is that multiplication by $v$ acts on them in a very simple way, and this is relevant
since the action of our streaming operator on the $k$th mode is multiplication by $ikv$. In fact, the 
reason for the simple nature of its action is very general and thus applies to generalizations 
of the Hermite functions. Since we use this below, we
explain the simple action from a general point of view, using only the fact that $M_T$ is even. 

Note that multiplication by $v$
is evidently self adjoint on $L^2(M_T^{-1})$. Also, for each $\ell$, $v\tilde g_\ell(v)$ is in the span of $\{\tilde g_0,\dots,\tilde g_{\ell+1}\}$. Hence,
for $m > \ell+1$
$$0 = \langle \tilde g_m, v\tilde g_\ell\rangle_{L^2(M_T^{-1})}  =  \langle \tilde g_\ell, v\tilde g_m\rangle_{L^2(M_T^{-1})}$$
from which we conclude that the $\ell,m$ matrix elements of multiplication by $v$ are zero for $|\ell - m| > 2$. Finally, by the
symmetry of $M_T$, the diagonal matrix elements are all zero. Hence, {\em in the Hermite basis, multiplication by $v$
is represented by a tridiagonal symmetric matrix that is zero on the main diagonal}. 
The operator ${\bf Q}$ is evidently diagonal in the Hermite basis. Hence the operator ${\bf L}_k := -ikv + {\bf Q}$ has a simple
tridiagonal structure.  We shall see that the matrix representing $ikv$ is
$$
ik\sqrt{T}\left(\begin{array}{cccc}
  0 &  \sqrt 1 &  0 &  \cdots \\
  \sqrt 1 & 0 &  \sqrt 2 &  0 \\
  0 &  \sqrt 2 & 0 &  \sqrt 3 \\
  \vdots &  0 &  \sqrt 3 & \ddots 
  \end{array}\right)
$$
while ${\bf Q} =\diag(0,\,-1,\,-1,\,\cdots)$.

The infinite tridiagonal matrix representing ${\bf L}_k = -ikv + {\bf Q}$ in the Hermite basis is still not easy 
to analyze directly. We cannot compute its eigenfunctions in closed form, and hence cannot apply formula \eqref{simpleP} 
to implement Lyapunov's method. 

However, we can do this for a related family of discrete velocity models, since then we are dealing with finite matrices. 
The discrete models, using the binomial approximation to the Gaussian distribution, are sufficiently close in structure to the 
continuous velocity BGK model that  they suggest an ansatz for the ${\bf P}$ operator that specifies the entropy function norm. 
In fact, a complete solution of a $2$-velocity model provides the essential hint for proving hypocoercivity of the continuous 
velocity BGK model. 

We shall present the details of this analysis in \S\ref{sec:43} below. Here, the above remark only serves as a 
motivation for our analysis of discrete velocity models, which are velocity discretizations of the BGK equation \eqref{bgk2}. 
We shall start with the two velocity case, and then discuss its generalization to $n$ velocities.

%%%%%%%%%%%%%%%%%%%%%%%
\subsection{A two velocity BGK model}
\label{sec:41}
In this section we revisit the following hyperbolic system, which can be considered as a kinetic equation with the two velocities $v=\pm\sigma$, and some parameter $\sigma>0$:
\begin{equation}\label{2vel}
  \partial_t{f_\pm}\pm \sigma \partial_x{f_\pm} =\pm \frac12 (f_- - f_+),\quad t\ge 0\,,
\end{equation}
for the distributions $f_\pm(x,t)$ of right- and left-moving particles, $2\pi$--periodic in $x$. The matrix of the interaction term on the r.h.s.\ has the form
$$
  \frac12\left(\begin{array}{cc}
 -1 & 1 \\
  1 & -1
\end{array}\right)\,,
$$
and hence \eqref{2vel} is also of BGK-form. Due to the conservation of the total mass $\int_0^{2\pi} \big(f_+(x,t)+f_-(x,t)\big) \,\d$ of \eqref{2vel}, its unique normalized steady state is $f_+^\infty=f_-^\infty=const=\frac1{4\pi} \int_0^{2\pi} \big(f_+^I(x)+f_-^I(x)\big) \,\d$.

This toy model (with the choice $\sigma=1$) was analyzed in \S1.4 of \cite{DoMoScH10} to illustrate the hypocoercivity method presented there. 
As for \eqref{bgk-modes}, we Fourier transform \eqref{2vel} in $x$ and expand it in the discrete velocity basis $\{\binom{1}{1},\,\binom{1}{-1}\}$. This yields for each mode $k\in\Z$ the following decoupled  ODE-system:
\begin{equation}\label{2vel-trans}
    \ddt{u_k}= -\CC_k\,u_k,\quad 
    \CC_k=\left(\begin{array}{cc}
    0 & ik\sigma \\
    ik\sigma & 1
\end{array}\right)\,,
\end{equation}
with $u_k(t)\in\C^2, \;k\in\Z$.
The matrices $-\CC_k$ have the eigenvalues $-\frac12 \pm \sqrt{\frac14-k^2\sigma^2}$ in the case $|k|\le \frac{1}{2\sigma}$ and $-\frac12 \pm i \sqrt{k^2\sigma^2-\frac14}$ in the case $|k|> \frac{1}{2\sigma}$.
Hence, as $t\to\infty$, $u_0(t)$ converges to an eigenvector of the 0-eigenvalue, i.e.\ $u_0^\infty=(f_+^\infty+f_-^\infty,\,0)^\top$, with the exponential rate $\lambda_0:=1$. All modes $u_k(t)$ with $k\ne0$ converge to $u_k^\infty=0$ with an exponential rate determined by the spectral gap of the matrix $\CC_k$. 
For simplicity we shall assume here that $\frac{1}{2\sigma}\not\in\N$. This avoids defective eigenvalues of the matrices $\CC_k$, but they could be included as discussed in Lemma 4.3 of \cite{ArEr14}. 
The spectral gap of the low modes (i.e.\ for $0<|k|< \frac{1}{2\sigma}$) is $\lambda_k:=\frac12 - \sqrt{\frac14-k^2\sigma^2}$, and it is $\lambda_k:=\frac12$ for the high modes. Hence, the exponential decay rate of the sequence of modes $\{u_k(t)\}_{k\in\Z}$ is given by the decay of the modes $k=\pm1$: $\lambda:=\min_{k\in\Z}\{\lambda_k\}=\Re\big(\frac12-\sqrt{\frac14-\sigma^2}\big)$. By Plancherel's theorem this is then also the convergence rate of $f(t)=(f_+(t),\,f_-(t))^\top$ towards the steady state $f^\infty=(f_+^\infty,\,f_-^\infty)^\top$.

The goal of entropy methods is to prove this exponential decay towards equilibrium, possibly with the sharp rate, by constructing an appropriate Lyapunov functional. In the hypocoercive method developed in \cite{DoMoScH10} the authors obtained, for the case $\sigma=1$ and the quadratic entropy, a decay rate bounded above by $\frac15$. But the sharp rate for this case is $\lambda=\frac12$. We shall now construct a refined Lyapunov functional that captures the sharp decay rate.

Following Lemma \ref{Pdefinition}(i) we introduce the positive definite transformation matrices $\P_0:=\II$, 
$$
  \P_k:=\left(\begin{array}{cc}
 4k^2\sigma^2 & -2ik\sigma \\
  2ik\sigma & 2-4k^2\sigma^2
\end{array}\right)\,,\quad\mbox{ for } \;\;0< |k|< \frac{1}{2\sigma}\,,
$$
and 
\begin{equation}\label{Pk-high}
  \P_k:=\left(\begin{array}{cc}
 1 & \frac{-i}{2k\sigma} \\
  \frac{i}{2k\sigma} & 1
\end{array}\right)\,,\quad\mbox{ for }  \;\;|k| > \frac{1}{2\sigma}\,.
\end{equation}
In the latter case, $\P_k$ is unique only up to a multiplicative constant, which is chosen here such that $\tr \P_k=n=2$.
%As in \eqref{Pnorm} 
We define the ``distorted'' vector norms for each mode $u_k$:
$$
  \|u_k\|_{\P_k} := \sqrt{\langle u_k,\,\P_ku_k\rangle}\,.
$$
Due to the ODE \eqref{2vel-trans} and the matrix inequality \eqref{matrixestimate} it satisfies
\begin{equation}\label{norm-ineq}
  \ddt \|u_k\|_{\P_k}^2 =- \langle u_k \,, (\CC_k^*\P_k+\P_k\CC_k) \,u_k \rangle
  \le-2\lambda_k\,\|u_k\|_{\P_k}^2\,,\quad k\in\Z\setminus \{0\}\,,
\end{equation}
and hence
\begin{equation}\label{uk-decay}
  \|u_k(t)-u_k^\infty\|_{\P_k}\le e^{-\lambda_k t}\|u_k(0)-u_k^\infty\|_{\P_k}\,,\quad t\ge0,\:\:k\in\Z\,.
\end{equation}

With this motivation we define the following norm as a Lyapunov functional for the sequence of modes:
\begin{equation}\label{normE}
  E\big(\{u_k\}_{k\in\Z}\big) := \sqrt{\sum_{k\in\Z} \|u_k\|_{\P_k}^2}\,.
\end{equation}
{}From \eqref{uk-decay} we obtain
$$
  E\big(\{u_k(t)-u_k^\infty\}\big) \le e^{-\lambda t} E\big(\{u_k(0)-u_k^\infty\}\big) \,,\quad t\ge0\,,
$$
with $\lambda=\min_{k\in\Z}\{\lambda_k\}$. Due to Plancherel's theorem, this is also a norm for the corresponding distributions $f=(f_+,\,f_-)^\top$:
$$
  E\big(\{u_k\}\big) = \|Bf\|_{L^2(0,2\pi;\R^2)}\,,
$$
where $B$ is a (nonlocal) bounded operator on $L^2(0,2\pi;\R^2)$ with bounded inverse. More precisely, $B=I+K$, where $K$ is a compact operator with $\|K\|<1$, since $\P_k\stackrel{|k|\to\infty}{\longrightarrow} \II$ (cf.\ \eqref{Pk-high}). This implies the sought-for exponential decay of $f(t)$ with sharp rate:
\begin{theorem}\label{th-periodicBGK}
Let $\frac1{2\sigma}\not\in\N$. Then the solution to \eqref{2vel} satisfies
$$
  \|f(t)-f^\infty\|_{L^2(0,2\pi;\R^2)} \le c\,e^{-\lambda t} \|f^I-f^\infty\|_{L^2(0,2\pi;\R^2)}\,,
  \quad t\ge0\,,
$$
with $\lambda=\Re\big(\frac12 - \sqrt{\frac14-\sigma^2} \big)$ and some generic constant $c>0$.
\end{theorem}

%Q: extendable to (linearized) x-periodic Broadwell model (see Inoue-Nishida, 1976) ?

%%%%%%%%%%%%%%%%%%%%%%%%%%%%%%%%%%%%%%%%%%%%%%%%%%%%%%%%%%%%%%%%%%%%%%%%%%%%%%%%%
\subsection{A multi-velocity BGK model}
\label{sec:42}

We now turn to a discrete velocity model analog of the linear BGK equation (\ref{bgk2}), and we shall establish its hypocoercivity.   Fixing unit temperature $T$, recall 
that as a consequence of the Central Limit Theorem,
the measure $M_1(v){\rm d}v$ is the (weak)  limit  of a sequence of discrete probability measures $\{\mu_n\}$ where
$$\mu_n  :=\sum_{j=0}^n 2^{-n} \binom{n}{j} \delta_{(2j - n)/\sqrt{n}}\ ,$$
where $\delta_y$ denotes the unit mass at $y\in \R$.  Each of the probability  measures $\mu_n$, $n\in \N$, has zero mean and unit variance. 

The Hermite polynomials have a natural discrete analog, namely the {\em Krawtchouk polynomials}.  A good reference containing proofs of all
of the facts we use below is the survey \cite{Col11}. (We are only concerned with a special family of the more general Krawtchouk polynomials
discussed in \cite{Col11}, namely the $s=2$ case in the terminology used there.)
The standard 
Krawtchouk polynomials of order $m$ are a set of $n+1$ polynomials $K_{n,m};\,m=0,...,n$ that are orthogonal with respect to the probability 
measure 
$$\omega_n  =\sum_{j=0}^n 2^{-n} \binom{n}{j} \delta_j\ ,$$
and are given by the following generating function:
\begin{equation}\label{kraw1}
(1+ t)^{n-v} (1-t)^v = \sum_{m=0}^n t^m K_{n,m}(v)\ .
\end{equation}
The leading coefficient of $K_{n,m}$ has the sign $(-1)^m$. 
One  has the orthogonality relations
\begin{equation}\label{orthong1}
\int_\R K_{n,m} K_{n,\ell}\,{\rm d}\omega_n = \begin{cases} \tbinom{n}{m} & m = \ell\ , \\ 0 & m \neq \ell\ .\end{cases}
\end{equation}
Then the {\em discrete Hermite polynomials} $H_{n,m}$ are defined by
\begin{equation}
H_{n,m}(v) := (-1)^m  \binom{n}{m}^{-1/2} K_{n,m} \left(\frac{n}{2}  + \frac{\sqrt{n}}{2} v\right) \qquad {\rm for}\quad m = 0,1,\dots, n\,;\ v\in\R\ .
\end{equation}
Then $\{H_{n,0}, \dots,H_{n,n}\}$ is the set  of $n+1$ polynomials that are orthogonal with respect to $\mu_n$, and hence are an 
orthonormal basis for $L^2(\R; \mu_n)$, and for each $m$ and $v$, $\lim_{n\to \infty}H_{n,m}(v) = \tfrac1{\sqrt{m!}} H_m(v)$. 
The analog of the crucial Hermite--recurrence relation (\ref{recur}) for the Krawtchouk polynomials is
$$(m+1)K_{n,m+1} = (n-2v)K_{n,m} - (n-m+1)K_{n,m-1}\ .$$
Rewriting this in terms of the discrete Hermite polynomials, one obtains
\begin{equation}\label{recur2}
v H_{n,m}(v) = \sqrt{m+1}\left(\frac{n-m}{n}\right)^{1/2}H_{n,m+1}(v) +  \sqrt{m}\left(\frac{n-m+1}{n}\right)^{1/2}H_{n,m-1}(v) \ .
\end{equation}
Notice that this reduces to \eqref{recur} in the limit $n\to\infty$ (up to the multiplication by the standard Gaussian). 

We are now ready to produce a discrete velocity analog of (\ref{bgk2}) in continuous $x$-space. The phase space is $\mathbf{T}^1\times [v_0,\dots,v_n]$
where the discrete velocity $v_j = (2j - n)/\sqrt{n}$. Our phase space density at time $t$ is a vector ${\bf f}(x,t)$ with $n+1$ non-negative entries $f_0(x,t), \dots, f_n(x,t)$,
such that 
$$
  \sum_{j=0}^n \left(\int_{\mathbf{T}^1} f_j(x,t){\rm d}x \right) = 1\ .
$$
We associate to ${\bf f}(x,t)$  the probability measure on the phase space given by
$$\sum_{j=0}^n f_j(x,t)  \delta_{(2j - n)/\sqrt{n}}\ .$$
The discrete unit Maxwellian (of order $n$) is the vector ${\bf m}  = 2^{-n}\Big( \tbinom{n}{0}, \tbinom{n}{1},\dots, \tbinom{n}{n}\Big)^\top$. 
Then the order $n$ discrete analog of (\ref{bgk2}) is the equation
\begin{equation}\label{n-velBGK}
  \partial_t{\bf f}(x,t)  + {\bf V} \partial_x {\bf f}(x,t) = {\bf m} \left(\sum_{j=0}^n f_j(x,t)\right)  - {\bf f}(x,t)\ , \quad t\ge0;\,x\in {\bf T}^1\,,
\end{equation}
with the $(n+1)\times(n+1)$ matrix ${\bf V}=\diag(v_0,...,v_n)$.
Proceeding as for \eqref{bgk-modes} yields the evolution equation for the spatial modes 
${\bf f}_k(t),\,k\in\Z$. 
Expanding ${\bf f}_k$ in the discrete Hermite basis $\{H_{n,m}(v_j);\,j=0,...,n\}_{m=0,...,n}$, we obtain for each $k$ the equation
$$
  \partial_t \hat {\bf f}_k + ik {\bf L}_1\hat {\bf f}_k = {\bf L}_2 \hat {\bf f}_k\,,
  \quad t\ge0;\,k\in\Z\,,
$$
where the vector $\hat {\bf f}_k(t)\in\C^{n+1}$ represents the basis coefficients of ${\bf f}_k(t)$.
As before  $\LL_2$ is the $(n+1)\times(n+1)$ matrix $\LL_2=\diag(0,\,-1,\,-1,\,\cdots)$, and $\LL_1$ is the symmetric tridiagonal matrix
whose diagonal entries are all zero, and whose superdiagonal sequence is given by 
$$ [\LL_1]_{m,m+1} = \sqrt{m+1} \left(\frac{n-m}{n}\right)^{1/2}\,; \qquad m =0 ,1,\dots, n-1\ .$$
For example, with $n=4$,
$$
  \LL_1= \left(\begin{array}{ccccc}
  0 &  1 &  0 &  0 & 0 \\
   1 & 0 &  \sqrt {3/2} &  0 &0 \\
  0 &  \sqrt {3/2} & 0 &  \sqrt {3/2} &0 \\
 0 &  0 &  \sqrt {3/2} & 0 & 1\\
 0 & 0 & 0 & 1 & 0
  \end{array}\right)\,.
$$

%Anton-B
Next we discuss the time decay of the solution to \eqref{n-velBGK} towards ${\bf f}^\infty={\bf m}$. We shall focus on the example with order $n=4$, but the other cases behave similarly.
Computing for the modes $k=\pm1$ the eigenvalues of $\mp i\LL_1 + \LL_2$ we find two complex pairs and one real eigenvalue $\lambda_0 = -0.526948302245121...$
which has the least negative real part, and hence determines the exponential decay rate of ${\bf f}_{\pm1}(t)$. This situation for higher $|k|$ is similar, but even better, with faster decay. 
To see this we write the eigenvalue equation for the matrices $-ik\LL_1 + \LL_2,\,k\in\Z$ as
$$
  h_0(\lambda):=\lambda(\lambda+1)^4 = -k^2(\lambda+1)^2(5\lambda+1)-k^4(4\lambda+\frac52)=:-k^2h_2(\lambda)-k^4h_4(\lambda)\,.
$$
The function $h_0$ is negative on $(-1,0)$, $-h_2$ on $(-\frac15,0]$ and $-h_4$ on $(-\frac58,0]$ (cf. Figure \ref{fig:h-curves}). For $k\ne0$, the function $k^2h_2(\lambda)+k^4h_4(\lambda)$ has exactly one real zero, $\tilde\lambda(k)$, and it is nonnegative on $[\tilde\lambda(k),0]$. For each fixed $k\in\Z$, the function $k^2h_2+k^4h_4$ is strictly increasing w.r.t.\ $\lambda$.
Hence, %for each fixed $k\in\Z$, 
the above eigenvalue equation has exactly one real zero $\lambda_0(k)$, and it lies in $(-\frac58,0]$. For each fixed $\lambda\in [\tilde\lambda(k),0]$, the function $k^2h_2+k^4h_4$ is strictly increasing w.r.t.\ increasing $|k|$. Hence, $\lambda_0(k)$ decreases monotonically (w.r.t.\ $|k|$) towards $-\frac58$.
\begin{figure}[ht!]
\begin{center}
 \includegraphics[scale=1]{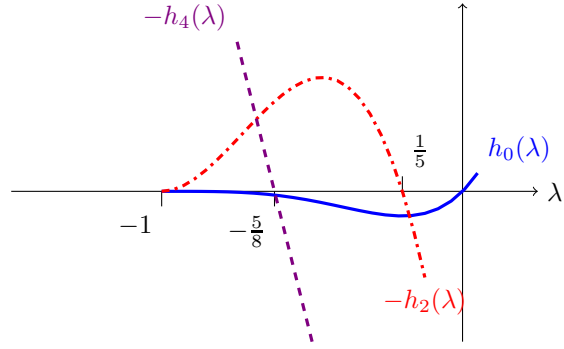}
 \caption{Functions appearing in the eigenvalue equation of $-ik\LL_1+\LL_2$; solid blue curve: $h_0(\lambda)$; red dash-dotted curve: $-h_2(\lambda)$; purple dashed line: $-h_4(\lambda)$.
  (colors only online)}
 \label{fig:h-curves}
\end{center}
\end{figure}

This proves that the 5 velocity model is hypocoercive, at least in the norm $E$ defined in \eqref{normE} (with the transformation matrices $\P_k$ now corresponding to $-\CC_k:=\mp ik\LL_1+\LL_2$). The sharp decay rate is given by $\lambda_0 = -0.526948302245121...$ . 

To establish a uniform-in-$k$ spectral gap was already cumbersome for the case $n=4$, and it becomes even more involved for larger $n$. In the following section we present a much simpler strategy, at the price of giving up sharpness of the decay rate.
But more importantly, that strategy will also be applicable for the continuous velocity case, which is represented by a tridiagonal ``infinite matrix''.

%%%%%%%%%%%%%%%%%%%%%%%%%%%%%%%%%%%%%%%%%%%%%%%%%%%%%%%%%%%%%%%%%555

\subsection{A continuous velocity BGK model}
\label{sec:43}

In this subsection we continue our discussion of the space-inhomogeneous BGK equation \eqref{bgk2} or, equivalently, \eqref{bgk-modes}. This will yield the proof of Theorem \ref{bgk-decay}.

Using the probabilists' Hermite polynomials,
\begin{equation} \label{hermite-polynom}
  H_m(v) := (-1)^m e^{\frac{v^2}{2}}\frac{{\rm d}^m}{{\rm d} v^m}e^{-\frac{v^2}{2}}\,,\quad m\in\N_0\,,
\end{equation}
we define the normalized Hermite functions
\begin{equation} \label{hermite-fct}
  g_m(v):=(2\pi m!)^{-1/2} H_m(v)\,e^{-\frac{v^2}{2}},\quad \mbox{and}\quad
  \tilde g_m(v):=\frac{1}{\sqrt T} g_m\big(\frac{v}{\sqrt T}\big)\,.
\end{equation}
They satisfy
$$
  \int_\R \tilde g_m(v)\tilde g_n(v) M_T^{-1}(v)\,{\rm d} v = \delta_{mn}\,
$$
and the recurrence relation
\begin{equation}\label{recur}
  v \,\tilde g_m(v) = \sqrt T\big[\sqrt{m+1}\, \tilde g_{m+1}(v) + 
  \sqrt{m} \,\tilde g_{m-1}(v)\big]\,.
\end{equation}
In the basis $\{\tilde g_m\}_{m\in\N_0}$ Equation \eqref{bgk-modes} becomes
\begin{equation}\label{bgk-hermite}
  \partial_t \hat {\bf f}_k +ik\sqrt T\,\LL_1 \hat {\bf f}_k = \LL_2 \hat {\bf f}_k\,,\quad t\ge0;\;k\in\Z\,.
\end{equation}
Here, the ``infinite vector'' $\hat {\bf f}_k(t)\in l^2(\N_0)$ is the representation of the function $f_k(v,t)\in L^2(\R;M_T^{-1})$ in the Hermite function basis, and the operators $\LL_1,\,\LL_2$ are represented by ``infinite matrices'' as
\begin{equation}\label{L1L2}
  \LL_1= \left(\begin{array}{cccc}
  0 &  \sqrt 1 &  0 &  \cdots \\
  \sqrt 1 & 0 &  \sqrt 2 &  0 \\
  0 &  \sqrt 2 & 0 &  \sqrt 3 \\
  \vdots &  0 &  \sqrt 3 & \ddots 
  \end{array}\right)
\,,\quad \LL_2=\diag(0,\,-1,\,-1,\,\cdots)\,.
\end{equation}

Next we shall prove the exponential decay of \eqref{bgk-hermite}, using a modified strategy compared to \S\ref{sec:42}. For the 5 velocity model there, it was possible (with some effort) to determine the sharp spectral gap of the matrices $-ik\LL_1+\LL_2$, uniform in all modes $k$. But since this seems not (easily) possible for the infinite dimensional case in \eqref{bgk-hermite}, we shall construct now approximate transformation matrices $\P_k$ that yield at least a (reasonable) lower bound on the spectral gap, and hence on the decay rate. For simplicity we set now $T=1$, as the temperature could be ``absorbed'' into the parameter $k$ by scaling.

Let $\A$ be an $(n+1)\times (n+1)$ tridiagonal matrix that is zero on the main diagonal.  That is $A_{i,j} = 0$
unless $ j = i+1$ or $i = j-1$.  We further suppose that $\A$ is real and symmetric, so that $\A$  is characterized
by the numbers $a_1,\dots,a_n$ where $a_j = A_{j-1,j}$.   Let $\B = {\rm diag}(0, -1,\dots,-1)$. Finally, for $k\in \Z$, consider the matrix
$-\CC_k := -ik\A + \B$. 

In the simplest case $n=1$ with $a=1$, we obtain 
$$\A = \left(\begin{array}{cc} 0 & 1\\ 1 & 0\end{array}\right)\,,  \qquad  \B = \left(\begin{array}{cc} 0 & \phantom{-}0\\ 0 & -1\end{array}\right) 
\quad{\rm  and}\quad   -\CC_k = \left(\begin{array}{cc} \phantom{-}0 & -ik\\ -ik &  -1\end{array}\right) \ .$$

\begin{comment}
Simple computations show that the eigenvalues $\lambda_\pm$ and eigenvectors $w_\pm$ of $\LL_k$ are 
$$\lambda_\pm = -\frac12(1 \pm \sqrt{1- 4k^2}) \qquad{\rm and}\qquad w_k = (-ik, \lambda_\pm)\ . $$
Since $|\lambda_\pm |^2 = k^2$,  
$$ w_\pm \otimes \overline{w_\pm}^\top =  \left(\begin{array}{cc} k^2 & ik\overline{\lambda_{pm}}\\ -ik\lambda_{pm} &  k^2\end{array}\right)\ ,  $$
and applying formula (\ref{simpleP}) with weights $b_\pm = 1/k^2$ yields
\begin{equation}\label{P2}
\P =  \left(\begin{array}{cc} 1 & -i/2k\\ i/2k & 1\end{array}\right)  \ .
\end{equation}
\end{comment}
For this matrix $\CC_k$, the transformation matrix $\P_k$ was already computed in \eqref{Pk-high} (with $\sigma=1$). For $k\ne0$ a simple computation yields
$$\CC^*_k \P_k + \P_k\CC_k = \P_k$$
so that with this choice of $\P_k$, Lyapunov's method yields exponential decay of the ODE-sequence $\ddt u_k=-\CC_k u_k,\,k\in\Z$ at the optimal rate $e ^{-t/2}$ (cf.\ \S\ref{sec:41}). \\

We now turn to $n> 1$. For $k\ne0$ define $\P_k$ to be the $(n+1)\times (n+1)$ matrix whose upper left $2\times 2$
block is  $\left(\begin{array}{cc} 1 & -i\alpha/k\\ i\alpha/k & 1\end{array}\right)$, where $0 < \alpha <k$ is a parameter to be chosen below,
and with the remaining diagonal entries being $1$, and all other entries being $0$.  Then the eigenvalues of $\P_k$ are $(k+\alpha)/k$, $1$ and 
$(k-\alpha)/k$, so that $\P_k$ is positive definite, and close to the identity for large $k$. 

We take $-\CC_k := -ik \A + \B$ as above.
Then
$$\CC_k^*\P_k + \P_k\CC_k = -ik(\A\P_k - \P_k\A) - (\B\P_k+\P_k\B)\ ,$$
and its upper left  $3\times 3$ block reads
\begin{equation}\label{3x3matrix}
\left(\begin{array}{ccc} 2a_1\alpha  & -i\alpha/k & a_2\alpha  \\ 
i\alpha/k & 2-2a_1\alpha & 0\\
a_2\alpha & 0& 2\end{array}\right)\ .  
\end{equation}
The lower right $(n-2)\times (n-2)$ block is $2$ times the identity, the off diagonal blocks are zero. 
In all of our finite dimensional approximations to \eqref{L1L2} we have $a_1 =1$.  The value of $a_2$ is different for the different discrete velocity models, but to simplify matters, we only present calculations for $a_2=\sqrt2$, which is the value for the limiting continuous velocity model.

The determinants of the upper left $2\times 2$ and $3\times 3$ blocks read, respectively,
$$\delta_2(\alpha,k) = 
  \alpha \left(4 - \left(4 + \frac {1}{k^2}\right)\alpha\right) \quad{\rm and}\quad 
 \delta_3(\alpha,k)= 4\alpha\left((\alpha-2)(\alpha-1)-\frac{\alpha}{2k^2}\right)\ .
$$
For each $k$, $\delta_3(\alpha,k)/\alpha$ has two positive roots, and is negative between them. Hence our matrix is positive definite 
when $\alpha$ lies between zero and the smaller positive root of $\delta_3(\alpha,k)/\alpha$. This root is least when $k=1$, when it has the value 
$\frac{7-\sqrt{17}}{4}\approx 0.719$.
Hence, by Sylvester's criterion, $\CC_k^*\P_k + \P_k\CC_k$ is positive definite for all $k\ne0$ if and only if 
$\alpha\in(0,\frac{7-\sqrt{17}}{4})\approx(0,0.719)$.
Note that also $\delta_2(\alpha,k)>0$ for these $\alpha,k$.

When $\alpha$ is in this range, our $3\times 3$ matrix \eqref{3x3matrix} has three positive eigenvalues $\lambda_1$, $\lambda_2$, $\lambda_3$ which we may take to be arranged  in increasing order. Then
$$\sqrt{\lambda_1} = \frac{\sqrt{\delta_3(\alpha,k)}}{\sqrt{\lambda_2\lambda_3}} \geq 
 \frac{2\sqrt{\delta_3(\alpha,k)}}{\lambda_2+\lambda_3} > \frac{\sqrt{\delta_3(\alpha,k)}}{2 }$$
 since the trace of our matrix is $4$.  Hence, the least eigenvalue $\lambda_1$  of our $3\times 3$ matrix satisfies
 $$\lambda_1 =\lambda_1(\alpha,k) \geq \frac14 \delta_3(\alpha,k)\ .$$
 Hence we choose $\alpha=\alpha_k$ to maximize $\delta_3(\alpha,k)$ between its first two roots. Its maximal value, $\delta_3(\alpha_k,k)$, depends on $k$,
 but it is easily seen to be least for $k=1$ with $\alpha_1=\tfrac13$. Simple computations and estimates then
 yield $\lambda_1(\alpha_k,k) \geq\tfrac14\delta_3(\alpha_1,1) =17/54$ for all $k$. 

Since we always take $\alpha< 1$, the largest eigenvalue of the matrix ${\bf P}_k$ (defined with $\alpha=\alpha_k$) is no more than $2$, uniformly in $k$. 
Hence 
\begin{equation}\label{C-estimate1}
  \CC_k^*\P_k +\P_k\CC_k  \geq   \frac{17}{54} \II \geq \frac{17}{108} \P_k
\end{equation}
uniformly in $k$. Thus in each Fourier mode, we at least have exponential decay (of a quadratic type entropy) at the rate $17/108$ (by proceeding as in \eqref{norm-ineq}).

Since this is also uniform in $n$, we obtain a bound for the continuous velocity model. Let the infinite matrix 
${\bf P}_k$ be the positive matrix using the optimal value of
$\alpha$ in the $k$th mode, and regarded as a bounded operator on $L^2(M_T^{-1})$ through its action on Hermite modes. Define the entropy function by
\begin{equation} \label{entropy:functional}
  e(f) :=  \sum_{k\in \Z} \big\langle (f_k(v) - M_1(v)), {\bf P}_k (f_k(v) - M_1(v))\big\rangle_{L^2(M_1^{-1})}\ .
\end{equation}
We obtain that, for solutions $f(t)$ of our BGK equation \eqref{bgk2} or, equivalently, \eqref{bgk-modes},
$$\frac{{\rm d}}{{\rm d}t}e(f(t)) \leq - \frac{17}{108} e(f(t))\ ,$$
giving exponential relaxation. 

The least eigenvalue of $\P_k$, $1- \alpha/k$, is at least $1- \frac{7 - \sqrt{17}}{4} > \frac14$ uniformly in $k$, and hence we have the inequality
$$e(f) \geq \frac14 \| f - M_1\|^2_{\mathcal H}\ ,$$
with $\mathcal H=L^2(\T^1\times \R;M_T^{-1}(v)\d[v])$.
\

The above method to establish exponential decay is simple to apply but does not give the sharp decay rate (it is off by a factor of about 9, as indicated by numerical results). Hence we shall now sketch how to improve on it. The essence of the above method is to use an ansatz for the transformation matrix $\P_k$, namely to use for its upper left $2\times2$ block the matrix from the 2 velocity case. Using instead larger blocks, will most likely improve the decay rate.

As a second alternative we shall now present an improvement of the crucial matrix inequality \eqref{C-estimate1},
but we shall keep the same ansatz for the matrix $\P_k$: In the inequality 
\begin{equation}\label{C-estimate}
  \CC_k^* \P_k+\P_k\CC_k- 2\mu \P_k \ge {\bf 0}
\end{equation}
we shall choose $\mu\in[0,1]$ as large as possible (related to the matrix inequality \eqref{matrixestimate}). The upper left $3\times 3$ block of this matrix on the l.h.s.\ reads 
$$\DD:=\left(\begin{array}{ccc} 2\alpha-2\mu  & -i\alpha(1-2\mu)/k & \sqrt2 \alpha  \\ 
i\alpha(1-2\mu)/k & 2-2\alpha-2\mu & 0\\
\sqrt2 \alpha & 0& 2-2\mu\end{array}\right)\ .  $$
We shall first derive strict inequalities on $\mu$ to obtain the positive definiteness of this matrix, using Sylvester's criterion.  From $\DD_{0,0}$ we deduce the first condition $0\le\mu<\alpha$. The determinant of the upper left $2\times 2$ block reads
$$
  \delta_2(\mu;\alpha,k)=4(\alpha-\mu)(1-\alpha-\mu)-\frac{\alpha^2}{k^2} (1-2\mu)^2\,.
$$
Since the last term increases with $|k|$, it suffices to consider $\delta_2$ for $k=1$. Next we want to establish the positivity of 
$$
  \frac{\delta_2(\mu;\alpha,1)}{4(1-\alpha^2)} = \mu^2-\mu + \alpha \frac{1-5\alpha/4}{1-\alpha^2}\,.
$$
The zero order term of this quadratic polynomial is positive on the relevant $\alpha$--interval $(0,\frac{7-\sqrt{17}}{4})\subset(0,\frac45)$, taking its maximum value $\frac14$ at $\alpha=\frac12$. For that limiting case, the r.h.s.\ reads $(\mu-\frac12)^2$, and for $0<\alpha<\frac12$,  $\delta_2(\mu;\alpha,1)$ always has a zero in the interval $(0,\frac12)$. This discussion yields the second condition $0\le\mu<\frac12$, related to $\alpha<\frac12$. 

Next we consider the positivity of the determinant of the upper left $3\times 3$ block, which reads
$$
  \delta_3(\mu;\alpha,k)=8(1-\mu)(\alpha-\mu)(1-\alpha-\mu)-4\alpha^2(1-\alpha-\mu)-2\frac{\alpha^2}{k^2} (1-\mu)(1-2\mu)^2\,.
$$
For the same reason as before, we only have to consider the case $k=1$. For the resulting cubic polynomial in $\mu$ we want to find its largest zero in the interval $[0,\frac12]$ w.r.t.\ the parameter $\alpha\in[0,\frac12]$. By numerical inspection we find that $\alpha_0\approx0.4684$ yields 
$\delta_3(\mu;\alpha_0,1)\ge0$ for $\mu\in[0,0.273796...]$. This yields the third condition on $\mu$ and shows that the matrix inequality \eqref{C-estimate} holds with $\mu_0:=0.273796...$, uniformly in $k\ne0$.
This somewhat more involved discussion shows that the decay rate can be improved to $2\mu_0\approx 0.547592$. This finishes the proof of Theorem~\ref{bgk-decay}.\\

\noindent
\underline{Remark:}
To appreciate the above decay rate $\mu_0$ (since $e(f)$ is a quadratic functional), we compare it to a numerical computation of the spectral gap of the ``infinite matrices'' $-ik\LL_1+\LL_2,\,k\in\Z$ from \eqref{L1L2}. To this end we  cut out the upper left $n\times n$ submatrix for large values of $n$. For increasing $n$ the spectral gap approaches $0.6973$. Hence our decay rate is off by only a factor of about $2.5$. If one desired a closer bound, one could work with a $\P$ matrix with a larger block, say $3\times 3$, in the upper left. 

%%%%%%%%%%%%%%%%%%%%%%%%%%%%%%%%%%%%%%%%%%%%%%%%%%%%%%%%%%%%%%%%%%%%%%%%%%%%%%%%%

\subsection{Linearized BGK equation}
\label{sec:44}

Next we shall analyze here the linearized BGK equation \eqref{linBGK:torus} for the perturbation $h(x,v,t) =f(x,v,t) -M_T(v)$.
We recall the definition of the normalized Hermite functions $\tilde{g}_m(v)$, $m\in\N_0$ from~\eqref{hermite-fct}
and give explicit expressions for 
$$
  \tilde{g}_0(v) = M_T(v) \quad{\rm and}\quad \tilde{g}_2(v) = \frac{ v^2- T}{\sqrt{2}T}\ M_T(v)\ .
$$
With this notation, \eqref{linBGK:torus} reads
$$
  h_t(x,v,t) + vh_x(x,v,t) = \left(\tilde {g}_0(v) -\frac{1}{\sqrt2} \tilde{g}_2(v)\right)\sigma(x,t)
  + \frac{1}{\sqrt2 T} \tilde{g}_2(v)\tau(x,t)-h(x,v,t)\ .
$$
Fourier transforming in $x$, as in \eqref{Fourier}, each spatial mode $h_k(v,t)$ evolves as
\begin{equation} \label{linBGK:Fourier} 
 \partial_t h_k +  ikv h_k = \tilde {g}_0(v)\sigma_k(t) 
      + \tilde{g}_2(v)\frac{1}{\sqrt2} \left(\frac{\tau_k(t)}{T}-\sigma_k(t)\right)-h_k\,,\quad k\in\Z;\,t\ge0\ .
\end{equation}
Here, $\sigma_k$ and $\tau_k$ denote the spatial modes of the $v$--moments $\sigma$ and $\tau$ defined in \eqref{f:perturb}.

Next we expand $h_k(\cdot,t)\in L^2(\R;M_T^{-1})$ in the orthonormal basis $\{\tilde{g}_m(v)\}_{m\in\N_0}$:
$$
  h_k(v,t)=\sum_{m=0}^\infty \hat h_{k,m}(t)\,\tilde{g}_m(v)\,,\quad\mbox{with }\quad
  \hat h_{k,m} = \langle h_k(v),\tilde{g}_m(v)\rangle_{L^2(M_T^{-1})}\ ,
$$
and the ``infinite vector'' $\hat {\bf h}_k(t) = ( \hat h_{k,0}(t),\,\hat h_{k,1}(t),\,...)^\top\in \ell^2(\N_0)$ contains all Hermite coefficients of $h_k(\cdot,t)$, for each $k\in\Z$. In particular we have
$$
  \hat h_{k,0} = \int_\R h_k(v) \tilde{g}_0(v) M_T^{-1}(v) \d[v] = \sigma_k
$$
and
$$
  \hat h_{k,2} = \int_\R h_k(v) \tilde{g}_2(v) M_T^{-1}(v) \d[v] = \frac{1}{\sqrt2} \left(\frac{\tau_k}{T}-\sigma_k\right)\ .
$$
Hence, \eqref{linBGK:Fourier} can be written equivalently as
$$
  \partial_t h_k(v,t) +  ikv h_k(v,t) = \tilde {g}_0(v) \hat h_{k,0}(t) 
      + \tilde{g}_2(v) \hat h_{k,2}(t)-h_k(v,t)\,,\quad k\in\Z;\,t\ge0\ .
$$
In analogy to \eqref{bgk-hermite}, its Hermite coefficients satisfy 
\begin{equation}\label{linap4}
\partial_t\hat {\bf h}_k(t) + ik\sqrt{T}\,\LL_1 \hat{\bf h}_k(t) = \LL_3 \hat{\bf h}_k(t)\,,\quad k\in\Z;\,t\ge0\ ,
\end{equation}
where the operators $\LL_1,\,\LL_3$ are represented by ``infinite matrices'' on $\ell^2(\N_0)$ by
\begin{equation*}%\label{L1L2}
  \LL_1= \left(\begin{array}{cccc}
  0 &  \sqrt 1 &  0 &  \cdots \\
  \sqrt 1 & 0 &  \sqrt 2 &  0 \\
  0 &  \sqrt 2 & 0 &  \sqrt 3 \\
  \vdots &  0 &  \sqrt 3 & \ddots 
  \end{array}\right)
\,,\quad \LL_3=\diag(0,\,-1,\, 0,\, -1,\, -1,\cdots)\,.
\end{equation*}

We remark that \eqref{linap4} simplifies for the spatial mode $k=0$. One easily verifies that the flow of \eqref{linBGK:torus} preserves \eqref{mloc2}, i.e. $\sigma_0(t)=0$, $\tau_0(t)=0$ $\forall t\ge0$. Hence, \eqref{linBGK:Fourier} yields 
$$
  \partial_t h_0(v,t) = -h_0(v,t)\,,\quad t\ge0\ .
$$
% \medskip
% 
% As in \S\ref{sec:43} we shall now consider only the case of normalized temperature $T=1$.

For $k\neq 0$, we note that the linearized BGK equation is very similar to the equation specified in (\ref{bgk-hermite}) 
and (\ref{L1L2}): The only difference is that $\LL_2$ is replaced by $\LL_3$, which has one more zero on the diagonal. 
Our treatment of $ik\sqrt{T}\LL_1 - \LL_2$ in the previous section suggests the form of the positive matrix $\P_k$ that will provide our Lyapunov functional in this case. 
We obtained the matrix $\P_k$ in that case by replacing four entries around the location of the 
zero in $\LL_2$ with the entries of $\left(\begin{array}{cc} 1 & -i\alpha/k\\ i\alpha/k & 1\end{array}\right)$, the matrix that provides the optimal $\P_k$ for the two-velocity model.  In the present case, we use two such matrices, one for each zero.  

For parameters $\alpha$ and $\beta$ to be chosen below, we define $\P_k$ to be the matrix that has
\begin{equation}\label{Ptwo}
\left(\begin{array}{cccc} 1  & -i\alpha/k  & 0 & 0 \\ 
i\alpha/k & 1  & 0 &0\\
0 & 0&1 & -i\beta/2k \\
0 & 0 & i\beta/2k & 1\end{array}\right)\  
\end{equation}
as its upper-left $4\times 4$ block, with all other entries being those of the identity. We define $-\CC_k = -ik\LL_1 + \LL_3$,
where, for the rest of this subsection, we use units in which $T=1$.

\begin{lemma} Choosing $\alpha = \beta =1/3$ in $\P_k$ uniformly in $|k|\in\N$, we have
\begin{equation}\label{goodbnd}
\CC_k^*\P_k + \P_k\CC_k  \geq 2\mu \P_k
\end{equation}
where 
\begin{equation}\label{goodbnd2}
\mu = 0.0206\ .
\end{equation}
\end{lemma}

\begin{proof}
We  compute that $\CC_k^*\P_k + \P_k\CC_k$ is twice the identity matrix whose upper left $5\times 5$ block is replaced by 
$$\DD_{k,\alpha,\beta} = 
\left(\begin{array}{ccccc}
2\alpha & -i\alpha/k & \sqrt{2}\alpha & 0 & 0 \\
i\alpha/k & 2-2\alpha  & 0 & -\beta/\sqrt{2} & 0\\
\sqrt{2}\alpha & 0 & \sqrt{3}\beta & -i\beta/2k & \beta\\
0 & -\beta/\sqrt{2} & i\beta/2k & 2-\sqrt{3}\beta & 0\\
0 & 0 & \beta & 0 & 2
  \end{array}\right) \,.
$$
% We  compute that $\CC_k^*\P_k + \P_k\CC_k$ is the matrix whose upper left $5\times 5$ block is
% $$\DD_{k,\alpha,\beta} = 
% \left(\begin{array}{ccccc}
% 2\alpha & -i\alpha/k & \sqrt{2}\alpha & 0 & 0 \\
% i\alpha/k & 2-2\alpha  & 0 & -\beta/\sqrt{2} & 0\\
% \sqrt{2}\alpha & 0 & \sqrt{3}\beta & -i\beta/2k & \beta\\
% 0 & -\beta/\sqrt{2} & i\beta/2k & 2-\sqrt{3}\beta & 0\\
% 0 & 0 & \beta & 0 & 2
%   \end{array}\right) \,
% $$
% and whose remaining entries are those of twice the identity matrix. 
We seek to choose $\alpha$ and $\beta$ to make this matrix 
positive definite.

For $1\leq j \leq 5$, let $\delta_j(k,\alpha,\beta)$ denote the determinant of the upper left $j\times j$ submatrix of $\DD_{k,\alpha,\beta}$. For $\alpha=\beta$, the first and third column of $\DD_{k,\alpha,\beta}$ have the common factor $\alpha$.
We then compute that
$$\delta_5(k,\alpha,\alpha) = \alpha^2p_5(\alpha,k)\,,$$
where $p_5(\alpha,k)$ is a cubic polynomial in $\alpha$ with coefficients depending on $k$:
\begin{align*}
p_5(\alpha,k) &= 16(\sqrt{3}-1) - \left[ 8\sqrt{3} + 16 +\frac{2+4\sqrt{3}}{k^2}\right] \alpha \\
&\quad + \left[34 -6\sqrt{3} + \frac{24k^2 + 1}{2 k^4}\right]\alpha^2 - \left[4\sqrt{3}-1 + \frac{\sqrt{3}}{k^2}\right]\alpha^3 \,.
\end{align*}
Next, we establish the bound
\begin{multline} \label{p_5}
 p_5(\alpha,k) \geq p_5(\alpha,1) \\ 
  = 16(\sqrt{3}-1) - (12\sqrt{3}+18)\alpha + (46.5- 6\sqrt{3})\alpha^2  - (5\sqrt{3} -1)\alpha^3 >0
\end{multline}
for $\alpha\in[0,\alpha_1]$ with $\alpha_1\approx 0.555$ and $|k|\in \N$.
To see the first inequality we consider
\[ p_5(\alpha,k) -p_5(\alpha,1) = \alpha (1-\tfrac1{k^2}) \varphi(\alpha,k) \]
with 
\[ \varphi(\alpha,k) := \sqrt{3} \alpha^2 - \big( \tfrac12 (1+\tfrac1{k^2}) +12 \big) \alpha +2 +4\sqrt{3}. \]
It satisfies $\varphi(\alpha,1)>0$ for $\alpha\in[0,\alpha_2]$ with $\alpha_2\approx 0.765$
and $\partial_k \varphi = \alpha/k^3$ for $\alpha\geq 0$ and $k\in\N$.
The r.h.s. of~\eqref{p_5} is easily seen to be monotone decreasing and evaluating it at $\alpha = 1/3$ and simplifying, we obtain
$p_5(\alpha,k) \geq 2.5$ for $\alpha \in [0,1/3]$. Finally, we then have
$$\delta_5(k,\alpha,\alpha) \geq 2.5\alpha^2$$
for $\alpha \in [0,1/3]$ and all $k\neq 0$. 
A similar but simpler analysis shows that for $j=1,2,3,4$, $\delta_j(k,\alpha,\alpha)>0$ for  $\alpha \in [0,1/3]$ and all $k\neq 0$. 

Thus, we choose $\alpha = \beta =1/3$ uniformly in $k$ and this makes $\DD_{k,\alpha,\beta}$ positive definite. Let
$\{\lambda_1,\lambda_2,\lambda_3,\lambda_4,\lambda_5\}$ be the eigenvalues of  $\DD_{k,1/3,1/3}$ arranged in 
increasing order.   We seek a lower bound on $\lambda_1$. 
Note that by the arithmetic-geometric mean inequality,

\begin{eqnarray*}
 \lambda_1 = \frac{\delta_5(k,1/3,1/3)}{\lambda_2 \lambda_3 \lambda_4 \lambda_5} &\geq& \delta_5(k,1/3,1/3) \left( \frac{ \lambda_2 + \lambda_3 + \lambda_4 + \lambda_5}{4}\right)^{-4}\\
 &\geq& 256\ \frac{\delta_5(k,1/3,1/3)}{(\tr[\DD_{k,1/3,1/3}])^4}\ .
\end{eqnarray*}
 Since $\tr[\DD_{k,\alpha,\beta}] =6$ independent of $k$, $\alpha$ and $\beta$, we finally obtain the bound $\lambda_1 \geq 0.0549$,
 and this means that, uniformly in $k\neq 0$, 
 \begin{equation}\label{Pbnd}
 \CC_k^*\P_k + \P_k\CC_k \geq 0.0549\ {\bf I}\ . 
 \end{equation}
 A simple computation shows that the eigenvalues of $\P_k$ are $1$, $1\pm 1/6k$, and $1\pm 1/3k$. Hence uniformly in $k$,
 \begin{equation}\label{Pbound}
 \frac23 {\bf I}  \leq \P_k \leq \frac43 {\bf I}\ .
 \end{equation}
Combining \eqref{Pbound} with \eqref{Pbnd} yields the result.
\smartqed \qed \end{proof}
To deduce the first statement of Theorem~\ref{linBGK-decay}
we consider a solution $h$ of \eqref{linBGK:torus}, and for $\gamma\geq 0$ the entropy functional $e_\gamma(f)$ defined by
\begin{equation}\label{gamma-entropy}
  e_\gamma(f) := \sum_{k\in \Z} (1+k^2)^{\gamma} \langle h_k(v), \P_k h_k(v)\rangle_{L^2(M_T^{-1})}\ ,
\end{equation}
with $f=M_1+h$.
Here the matrices $\P_0 = \II$ and $\P_k$ defined in \eqref{Ptwo} for $k\neq 0$
are regarded as bounded operators on $L^2(M_T^{-1})$.
Then 
\begin{equation}\label{e-inequality}
  \ddt e_\gamma(f) = -\sum_{k\in \Z} (1+k^2)^{\gamma} \langle h_k(v), (\CC_k^*\P_k + \P_k\CC_k) h_k(v)\rangle_{L^2(M_T^{-1})} 
   \leq - 0.0412 \,e_\gamma(f) \ , 
%   \leq - \tfrac1{25} e_\gamma(h) \ , 
\end{equation}
which implies~\eqref{locasstab5} and this finishes the proof of Theorem \ref{linBGK-decay}(a).

%We remark that this decay rate is too small by about a factor of 18, which was obtained by numerically computing the spectrum of $\CC_k$.
We note that the constant in \eqref{e-inequality}
is within a factor of 18 of what numerical calculation shows is best possible.  With more work, in particular not making the
simplifying assumption $\alpha = \beta$ in the definition of $\P$, and also employing some of the ideas in the final part of section \ref{sec:43}, one can still better  within this framework.
%%%%%%%%%%%%%%%%%%%%%%%%%%%%%%%%%%%%%%%%%%%%%%%%%%%%%%%%%%%%%%%%%%%%%%%%%%%%%%%%%

\subsection{Local asymptotic stability for the BGK equation}
\label{sec:45}

For $\gamma \ge 0$, let $H^\gamma(\mathbf{T}^1)$ be the Sobolev space consisting of  the completion  of smooth functions $\varphi$ on $\mathbf{T}^1$ in the Hilbertian norm 
$$\|\varphi\|_{H^\gamma}^2 :=   \sum_{k\in \Z} (1+k^2)^{\gamma}|\varphi_k|^2\ ,$$
where $\varphi_k$ is the $k$th Fourier coefficient of $\varphi$. Let $\mathcal{H}_\gamma$ denote the Hilbert space
$H^\gamma(\mathbf{T}^1)\otimes L^2(\R;M_T^{-1}(v)\d[v])$, where the inner product in $\mathcal{H}_\gamma$ is given by
$$\langle f,g\rangle_{\mathcal{H}_\gamma} = \int_{\mathbf{T}^1} \int_\R \overline{f}(x,v) \left[\left(1 - \partial_x^2\right)^{\gamma} g(x,v) \right]M_T^{-1}(v){\rm d}v {\rm d}\tilde x \ , $$
where $\d[\tilde x]$ denotes the normalized Lebesgue measure on $\T^1$.

Then $\mathcal{H}_0$ is simply the weighted space $L^2(\mathbf{T}^1\times\R;M_T^{-1}(v)\d[v])$
and, for all $\gamma\geq 0$,  ${\bf Q}$ is self-adjoint on $\mathcal{H}_\gamma$. %, ${\bf L}f^\infty = 0$.

Let $\rho$, $P$, $\sigma$ and $\tau$ be defined in terms of a density $f$ as in (\ref{f:perturb}).  For all $\gamma$,
$\|\sigma\|_{H^\gamma}^2  = \langle \sigma M_T, f - f^\infty\rangle_{\mathcal{H}_\gamma}$.
Then by the Cauchy-Schwarz inequality,
\begin{equation}\label{locbound1}
\|\sigma\|_{H^\gamma}^2 \leq  \|\sigma M_T\|_{\mathcal{H}_\gamma} \| f-f^\infty\|_{\mathcal{H}_\gamma}  =
\|\sigma\|_{H^\gamma} \| f-f^\infty\|_{\mathcal{H}_\gamma} \ .
\end{equation} 
Likewise, $\|\tau\|_{H^\gamma}^2  = \langle \tau v^2 M_T, f - f^\infty\rangle_{\mathcal{H}_\gamma}$, and by 
the Cauchy-Schwarz inequality,
\begin{equation}\label{locbound2}
\|\tau\|_{H^\gamma}^2 \leq  \|\tau v^2 M_T\|_{\mathcal{H}_\gamma} \| f-f^\infty\|_{\mathcal{H}_\gamma}  =
\sqrt{3}T\|\tau\|_{H^\gamma} \| f-f^\infty\|_{\mathcal{H}_\gamma} \ .
\end{equation} 
For  $\gamma> 1/2$, functions in $H^\gamma$ are H\"older continuous, and the $H^\gamma$ norm controls
their supremum norm.   Combining this with the estimates proved above, we see that for all  $\gamma> 1/2$, there is a finite constant
$C_\gamma$ such that the pressure and density satisfy
\begin{equation}\label{locbound3}
\|\sigma\|_\infty=\|\rho - 1\|_\infty \leq C_\gamma  \| f-f^\infty\|_{\mathcal{H}_\gamma} \qquad{\rm and}\qquad 
\|\tau\|_\infty=\|P - T\|_\infty \leq C_\gamma  \| f-f^\infty\|_{\mathcal{H}_\gamma} \ .
\end{equation}
Using these estimates it is a simple matter to control the approximation in (\ref{taylor}). For $s\in [0,1]$ and $(x,v) \in \mathbf{T}^1\times\R$, define
$$F(s,x,v) := \frac{(1+s\sigma(x))^{3/2}}{\sqrt{2\pi (T+s\tau(x))}}e^{-v^2(1+s\sigma(x))/2(T+s\tau(x))}\ ,$$
so that the gain term in the linearized BGK equation \eqref{linBGK:torus} is
${\displaystyle \partial_s F(0,x,v)}$.
In this notation,
\begin{eqnarray*}
%&&{\bf R}_f(x,v) :=\\
&&R_f(x,v) :=\\
&&\qquad M_f(x,v) - M_T(v) - \left[ \left( \frac32 - \frac{v^2}{2T}\right)M_T(v)\sigma(x)
      + \left( -\frac{1}{2T}+ \frac{v^2}{2T^2}\right)M_T(v)\tau(x)\right] \\
      &&\qquad= \int_0^1 \left[ \partial_s F(s,x,v) - \partial_s F(0,x,v)\right]{\rm d}s
      = \int_0^1 \int_0^s \left[ \partial_s^2 F(r,x,v)\right] {\rm d}r{\rm d}s\ .
   \end{eqnarray*}
% Franz 16.1.2016
We compute
\begin{eqnarray} \label{F2}
 &&\partial_s^2 F(s,x,v)\\
   &&  = \frac{\tau-T\sigma}{(1+s\sigma)^2} \left[ - \frac{3\sigma}{4\theta_s} 
       + \bigg( \frac32 v^2\sigma + \frac34 \tau\bigg)\frac1{\theta_s^2}
       - \bigg( \frac14 v^4\sigma + \frac32 v^2\tau\bigg)\frac1{\theta_s^3}
       + \frac{v^4\tau}{4\theta_s^4}
       \right] M_{\theta_s}(v) \nonumber
\end{eqnarray}
with the notations $\theta_s:=\frac{T+s\tau}{1+s\sigma}$ and $M_{\theta_s}(v) := \frac1{\sqrt{2\pi \theta_s}} e^{-v^2/(2\theta_s)}$. Note that the r.h.s.~of \eqref{F2} is of the order $\mathcal O(\sigma^2+\tau^2)$, which will be related to $\mathcal O((f-f^\infty)^2)$ due to the estimates \eqref{locbound1}-\eqref{locbound2}.

Simple but cumbersome calculations now show that  
if $\gamma> 1/2$ and $\norm{f-f^\infty}_{\mathcal{H}_\gamma}$ is sufficiently small, 
then there exists a finite constant $\tilde C_{\gamma,T}$ depending only on $\gamma$ and $T$
such that for all $s\in [0,1]$,
\begin{equation}\label{F}
\left\Vert \partial_s^2 F(s,x,v) \right \Vert_{\mathcal{H}_\gamma}   \leq \tilde C_{\gamma,T}
\| f-f^\infty\|_{\mathcal{H}_\gamma}^2\ ,
\end{equation}
%$$\left\Vert \partial_s F(s,x,v) - \partial_s F(0,x,v)\right \Vert_{\mathcal{H}_\gamma}   \leq \tilde C_{\gamma,T}
%\| f-f^\infty\|_{\mathcal{H}_\gamma}^2\ ,$$
and hence
\begin{equation}\label{R-estimate}
%  \|{\bf R}_f\|_{\mathcal{H}_\gamma}\le \tilde C_{\gamma,T}
  \|R_f\|_{\mathcal{H}_\gamma}\le \tilde C_{\gamma,T}
\| f-f^\infty\|_{\mathcal{H}_\gamma}^2\  .
\end{equation}
[The calculations are simplest for non-negative integer $\gamma$, in which case the Sobolev norms can be calculated by differentiation. 
For $\gamma>1/2$ and sufficiently small $\norm{f-f^\infty}_{\mathcal{H}_\gamma}$,
the estimates (\ref{locbound3}) ensure for all $s\in [0,1]$
the boundedness of $0<\epsilon<\norm{1+s \sigma}_\infty\ ,\, \norm{T+s \tau}_\infty<\infty$ for some fixed $\epsilon>0$
and the $L^2(\R;M_T^{-1}(v)\d[v])$-integrability of 
\[ e^{-v^2(1+s\sigma(x))/2(T+s\tau(x))} \leq  e^{-v^2/3T} \quad \text{for all } x\ . \]
%We compute
%\begin{equation} \label{F}
% \partial_s F(s,x,v)
%     = \bigg(\frac32 - \frac{v^2}{2\theta_s}\bigg) M_{\theta_s}(v)\sigma 
%       + \bigg( - \frac1{2\theta_s} +  \frac{v^2}{2\theta_s^2} \bigg) M_{\theta_s}(v)\tau   
%\end{equation}
%with the notations $\theta_s:=\frac{T+s\tau}{1+s\sigma}$ and $M_{\theta_s}(v) := \frac1{\sqrt{2\pi \theta_s}} e^{-v^2/(2\theta_s)}$. Using the triangle inequality in $\left\Vert \partial_s F(s,x,v) - \partial_s F(0,x,v)\right \Vert_{\mathcal{H}_\gamma}$ and the estimates \eqref{locbound1}-\eqref{locbound2}, the appearance of $\sigma$, $\tau$ and its derivatives as multiplicative factors (in contrast to $1+s\sigma$ and $T+s\tau$) translate into a factor $\norm{f-f^\infty}_{\mathcal{H}_\gamma}$. 
%Then \eqref{F} shows that in the estimate of $\left\Vert \partial_s F(s,x,v) - \partial_s F(0,x,v)\right \Vert_{\mathcal{H}_\gamma}$ depends at least quadratically on $\norm{f-f^\infty}_{\mathcal{H}_\gamma}$,
%whereas higher powers can be absorbed into the constant $\tilde C_{\gamma,T}$.
In \eqref{F}, higher powers of $\norm{f-f^\infty}_{\mathcal{H}_\gamma}$ (arising due to derivatives of $\sigma$ and $\tau$) can be absorbed into the constant of the quadratic term.] 
%between the neighboring integer-order Sobolev norms.)
 
Now let $f$ be a solution of the BGK equation~\eqref{bgk:torus} with constant temperature $T=1$ 
%  \begin{equation} \label{bgk:torus2} 
%  f_t(x,v,t) +  v\ f_x(x,v,t) = M_f(x,v,t) -  f(x,v,t)\ , \qquad t\geq 0\ , 
%  \end{equation}
%  already introduced in (\ref{bgk:torus2})
 and define $h(x,v,t) := f(x,v,t) - M_T(v)$ as in the introduction. Now define the linearized BGK operator
\begin{equation}%\label{locstab21}
\Q_2 h(x,v,t)  :=  \left( \frac32 - \frac{v^2}{2T}\right)M_T(v)\sigma(x)
      + \left( -\frac{1}{2T}+ \frac{v^2}{2T^2}\right)M_T(v)\tau(x)  - h(x,v,t)
\end{equation}
where of course $\sigma$ and $\tau$ are determined by $f$, and hence $h$. %, and define
%\begin{equation}%\label{locstab21}
%{\bf R}_f(x,v,t) :=    M_f(x,v,t) - M_T(v) - \Q_2 h(x,v,t)\ .
%\end{equation}
Then the nonlinear BGK equation (\ref {bgk:torus}) becomes
\begin{equation} \label{bgk:torus3} 
% h_t(x,v,t) +  v\ h_x(x,v,t) =  \Q_2 h(x,v,t)  +  {\bf R}_f(x,v,t)\ , \qquad t\geq 0\ , 
 h_t(x,v,t) +  v\ h_x(x,v,t) =  \Q_2 h(x,v,t)  +  R_f(x,v,t)\ , \qquad t\geq 0\ , 
\end{equation}
which deviates from the linearized BGK equation \eqref{linBGK:torus} only by the additional term $R_f$.

It is now a simple matter to prove local asymptotic stability. We shall use here exactly the entropy functional $e_\gamma(f)$ defined in \eqref{gamma-entropy} with $f=M_1+h$.
%We introduce an entropy functional $e_\gamma(h)$ by
%$$e_\gamma(h) := \sum_{k\in \Z} (1+k^2)^{\gamma} \langle h_k(v), \P_k h_k(v))\rangle_{L^2(M_T^{-1})}\ ,$$
%where $\P_0=\II$ is the identity, and for $k\neq 0$, $\P_k$ is the bounded operator on $L^2(M_T^{-1})$ defined as in the previous subsection. 
Now assume that $h$ solves (\ref{bgk:torus3}). To compute $\frac{{\rm d}}{{\rm d}t} e_\gamma(f)$ we use the inequality \eqref{e-inequality} for the drift term and for $\Q_2 h$ in \eqref{bgk:torus3}, as well as $\|\P_k\|\le\frac43$ and \eqref{R-estimate} for the term $R_f$.
This yields
\begin{equation} \label{e-decay1}
\frac{{\rm d}}{{\rm d}t} e_\gamma(f) \leq  -%\frac{1}{25}
0.0412 \,e_\gamma(f)  + \frac83 \tilde C_{\gamma,T}\| h\|_{\mathcal{H}_\gamma}^3\ ,
\end{equation}
(if $\|h\|_{\mathcal{H}_\gamma}$ is small enough)
where we have used the fact that $h = f-f^\infty$.
Then since 
$$\frac23 e_\gamma(f)  \leq \| h\|_{\mathcal{H}_\gamma}^2 \leq \frac43 e_\gamma(f)\ ,$$
which is simply a restatement of (\ref{Pbound}), 
it is now simple to complete the proof of Theorem~\ref{linBGK-decay}(b): 
Estimate \eqref{e-decay1} shows that there is a $\delta_\gamma>0$ so that if the initial data $f^I(x,v)$ satisfies $\|f^I- f^\infty\|_{\mathcal{H}_\gamma} < \delta_\gamma$,
then the solution $f(t)$ satisfies
%$$\|f(t) - f^\infty\|_{\mathcal{H}_\gamma} \leq e^{-t/50}\delta_\gamma\ .$$
$$e_\gamma(f(t))  \leq e^{-t/25}e_\gamma(f^I)\ .$$
Here we used that the linear decay rate in \eqref{e-decay1} is slightly better than $\frac{1}{25}$, to compensate the nonlinear term.

\medskip
We expect that the strategy from this section can be adapted also to nonlinear kinetic Fokker-Planck equations; this will be the topic of a subsequent work.

%%%%%%%%%%%%%%%%%%%%%%%%%%%%%%%%%%%%%%%%%%%%%%%%%%%%%%%%%%%%%%%%%%%%%%%%%%%%%%%%%

\begin{acknowledgement}
The second author (AA) was supported by the FWF-doctoral school ``Dissipation and dispersion in non-linear partial differential equations''.
The third author (EC) was partially supported by U.S. N.S.F. grant DMS 1501007.
\end{acknowledgement}
%

%%%%%%%%%%%%%%%%%%%%%%%%%%%%%%%%%%%%%%%%%%%%%%%%%%%%%%%%%%%%%%%%%%%%%%%%%%%%%%%%%

\end{document}